\documentclass[12pt]{amsart}
\copyrightinfo{2010}{Luke Oeding}
\usepackage{url,amstext,amsfonts,amssymb,amscd,amsbsy,amsmath,amsthm,verbatim,mathrsfs}
\usepackage{ifthen, fullpage, mathrsfs, cite,hyperref}
\usepackage{color}

\newtheorem{theorem}{Theorem}[section]
\newtheorem{conjecture}[theorem]{Conjecture}
\newtheorem{lemma}[theorem]{Lemma}
\newtheorem{prop}[theorem]{Proposition}
\newtheorem{obs}[theorem]{Observation}

\newtheorem{cor}[theorem]{Corollary}
\newtheorem{question}[theorem]{Question}

\theoremstyle{definition}

\newtheorem{example}[theorem]{Example}

\theoremstyle{remark}
\newtheorem{remark}[theorem]{Remark}
\newcommand{\SL}{\operatorname{SL}}
\newcommand{\GL}{\operatorname{GL}}

\newcommand{\ZZ}{\mathbb Z}

\newcommand{\CC}{\mathbb C}
\newcommand{\PP}{\mathbb P}
\newcommand{\V}{\mathcal{V}}
\newcommand{\I}{\mathcal{I}}
\newcommand{\A}{\mathcal{A}}
\newcommand{\G}{\left( \SL(2)^{\times n}\right) \ltimes \mathfrak{S}_{n}}
\newcommand{\tdt}{\times\dots\times}
\newcommand{\otdot}{\otimes\dots\otimes}

\newcommand{\ie}{\emph{i.e. }}
\DeclareMathOperator{\Sym}{Sym}
\DeclareMathOperator{\adj}{adj}

\begin{document} 
% \title[short text for running head]{full title}
\title[Principal Minors of Symmetric Matrices]{Set-Theoretic Defining Equations of the Variety of Principal Minors of Symmetric Matrices}

% \author[short version for running head]{name for top of paper}
\author{Luke Oeding}
\address{Dipartimento di Matematica ``U. Dini'' \\ Universit\`a degli Studi di Firenze\\ Viale Morgagni 67/A\\ 50134 Firenze, Italy }
\email{oeding@math.unifi.it}
\thanks{This material is based upon work supported by the National Science Foundation under Award No. 0853000: International Research Fellowship Program (IRFP),
and U. S. Department of Education grant Award No. P200A060298: Graduate Fellowships for Ph.D. Students of Need in Mathematics (GAANN)}

%    \subjclass is required.

% 14L30 Group actions on varieties or schemes (quotients)
%13A50 Actions of groups on commutative rings; invariant theory
% 14M12 Determinantal varieties
% 20G05 Representation theory
% 15A72 Vector and tensor algebra, theory of invariants
% 15A69 Multilinear algebra, tensor products
% 15A29 Inverse problems

\date{\today}

%    Abstract is required.
\begin{abstract}
The variety of principal minors of $n\times n$ symmetric matrices, denoted $Z_{n}$, is invariant under the action of a group $G\subset \GL(2^{n})$ isomorphic to $\G$.  We describe an irreducible $G$-module of degree $4$ polynomials constructed from Cayley's $2 \times 2 \times 2$ hyperdeterminant and show that it cuts out $Z_{n}$ set-theoretically.  This solves the set-theoretic version of a conjecture of Holtz and Sturmfels.  Standard techniques from representation theory and geometry are explored and developed for the proof of the conjecture and may be of use for studying similar $G$-varieties.
\end{abstract}
\maketitle
%\tableofcontents

%%%%%%%%%%% Introduction  %%%%%%%%%%%%
\section{Introduction}\label{sec:Intro}

The problem of finding the relations among principal minors of a matrix of indeterminants dates back (at least) $1897$ when Nanson \cite{Nanson} found relations among the principal minors of an arbitrary $4\times 4$ matrix.  In $1928$ Stouffer \cite{Stouffer1928} found an expression for the determinant of a matrix in terms of a subset of its principal minors.  
Griffin and Tsatsomeros \cite{GriffTsat2} point out that the number of independent principal minors was essentially known to Stouffer in 1924, \cite{Stouffer1924, Stouffer1928}.  In fact, Stouffer \cite{Stouffer1928} claims that this result was already known to MacMahon in 1893 and later by Muir.
Subsequently, interest in the subject seems to have diminished, however
much more recently, there has been a renewed interest in the relations among principal minors and their application to matrix theory, probability, statistical physics and spectral graph theory. 

In response to questions about principal minors of \emph{symmetric} matrices, Holtz and Sturmfels \cite{HSt} introduced the algebraic variety of principal minors of symmetric $n\times n$ matrices (denoted $Z_{n}$ herein -- see Section~\ref{sec:defs} for the precise definition) and asked for generators of its ideal. 
In the first nontrivial case, \cite{HSt} showed that $Z_{3}$ is an irreducible hypersurface in $\PP^{7}$ cut out by a special degree four polynomial, namely Cayley's hyperdeterminant of format $2 \times 2 \times 2$.  In the next case they showed (with the aid of a computer calculation) that the ideal of $Z_{4}$ is minimally generated by $20$ degree four polynomials, but only $8$ of these polynomials are copies of the hyperdeterminant constructed by natural substitutions. The other 12 polynomials were separated into classes based on their multidegrees. This was done in a first draft of \cite{HSt}, and at that point, the geometric meaning of the remaining polynomials and their connection to the hyperdeterminant was still somewhat mysterious. Because of the symmetry of the hyperdeterminant, Landsberg suggested to Holtz and Sturmfels the following:
\begin{theorem}[{\cite[Theorem 12]{HSt}}]\label{thm:symmetry} The variety $Z_{n}$ is invariant under the action of 
\[\G.\]
\end{theorem}
It should be noted that Borodin and Rains \cite{BorodinRains} found a similar result for two other cases; when the matrix is not necessarily symmetric and for a Pfaffian analog. In \cite{oeding_thesis}, we showed that $Z_{n}$ is a linear projection of the well-known Lagrangian Grassmannian,  which can also be interpreted as the variety of all minors of a symmetric matrix.  We used this projection to give a geometric proof of Theorem~\ref{thm:symmetry}.

In \cite{HSt}, the span of the $\G$-orbit of the $2 \times 2 \times 2$ hyperdeterminant is named \emph{the hyperdeterminantal module} (denoted $HD$ herein -- see Section~\ref{sec:tensorreps}).  It was then understood -- and included in the final version of \cite{HSt} -- that the $20$ degree four polynomials are a basis of the hyperdeterminantal module when $n=4$.  This interpretation led to the following:
\begin{conjecture}[{\cite[Conjecture 14]{HSt}}]
The prime ideal of the variety of principal minors of symmetric matrices, is generated in degree four by the hyperdeterminantal module for all $n \geq 3$.
\end{conjecture}

While the first two cases of the conjecture ($n=3,4$) were proved using a computer, the dimension of the hyperdeterminantal module and the number of variables both grow exponentially with $n$ and this renders computational methods ineffective already in the next case $n=5$, for which the hyperdeterminantal module has a basis of $250$ degree $4$ polynomials on $32$ variables. Our point of departure is the use of the symmetry of $Z_{n}$ via tools from representation theory and the geometry of $G$-varieties.

The main purpose of this work is to solve the set-theoretic version of the Holtz--Sturmfels conjecture:   (See Example~\ref{rmk:HDweight} for the representation-theoretic description of the hyperdeterminantal module in terms of Schur modules used in the following statement.)
\begin{theorem}[Main Theorem]\label{thm:Main Theorem}
The variety of principal minors of symmetric $n\times n$ matrices, $Z_{n}$, is cut out set-theoretically by the hyperdeterminantal module, which is the irreducible $\G$-module of degree $4$ polynomials
\[HD = S_{(2,2)} S_{(2,2)} S_{(2,2)} S_{(4)} \dots  S_{(4)}.\]
\end{theorem}

The set-theoretic result is sufficient for many applications related to principal minors of symmetric matrices.  In particular, set-theoretic defining equations of $Z_{n}$ are necessary and sufficient conditions for a given vector of length $2^{n}$ to be expressed as the principal minors of a symmetric matrix. We state this practical membership test as follows:
\begin{cor}\label{cor}
Suppose $w=[w_{[i_{1},\dots,i_{n}]}] \in \CC^{2^{n}}$ with $i_{j}\in \{0,1\}$.  Then $w$ represents the principal minors of a symmetric $n \times n$ matrix if and only if $w$ and all images under changes of coordinates by $\G$ are zeros of Cayley's $2 \times 2 \times 2$ hyperdeterminant
\begin{multline*}  
 (w_{I_{[0,0,0]}})^{2} (w_{I_{[1,1,1]}})^{2} +(w_{I_{[1,0,0]}})^{2} (w_{I_{[0,1,1]}})^{2}
+(w_{I_{[0,1,0]}})^{2} (w_{I_{[1,0,1]}})^{2} +(w_{I_{[0,0,1]}})^{2} (w_{I_{[1,1,0]}})^{2}
\\
-2 w_{I_{[0,0,0]}} w_{I_{[1,0,0]}} w_{I_{[0,1,1]}} w_{I_{[1,1,1]}}
-2 w_{I_{[0,0,0]}} w_{I_{[0,1,0]}} w_{I_{[1,0,1]}} w_{I_{[1,1,1]}} 
-2 w_{I_{[0,0,0]}} w_{I_{[0,0,1]}} w_{I_{[1,1,0]}} w_{I_{[1,1,1]}}
\\
-2 w_{I_{[1,0,0]}} w_{I_{[0,1,0]}} w_{I_{[0,1,1]}} w_{I_{[1,0,1]}}
-2 w_{I_{[1,0,0]}} w_{I_{[0,0,1]}} w_{I_{[0,1,1]}} w_{I_{[1,1,0]}}
-2 w_{I_{[0,1,0]}} w_{I_{[0,0,1]}} w_{I_{[1,0,1]}} w_{I_{[1,1,0]}} 
\\
+4 w_{I_{[0,0,0]}} w_{I_{[0,1,1]}} w_{I_{[1,0,1]}} w_{I_{[1,1,0]}}
+4 w_{I_{[0,0,1]}} w_{I_{[0,1,0]}} w_{I_{[1,0,0]}} w_{I_{[1,1,1]}}
,\end{multline*}
where $I_{[i_{1},i_{2},i_{3}]} = [i_{1},i_{2},i_{3},0, \dots ,0]$ for $i_{j}\in \{0,1\}$.
\end{cor}

A second, unifying purpose of this work is to study $Z_{n}$ as a prototypical (non-homo\-geneous) $G$-variety.  
We aim to show the use of standard constructions in representation theory and geometry, and to further develop general tools for studying geometric and algebraic properties of such varieties.   
We anticipate these techniques will be applicable to other $G$-varieties in spaces of tensors such as those that arise naturally in computational complexity \cite{LBull, Burgisser}, signal processing \cite{CGLM, ComoR06, DD08}, and algebraic statistics \cite{PachterSturmfels, AllmanRhodes08} (see also \cite{LandsbergTensorBook} for a unified presentation of the use of geometry and representation theory in these areas), and especially to the case of principal minors of arbitrary matrices studied by Lin and Sturmfels, \cite{LinSturmfels} and Borodin and Rains \cite{BorodinRains}.  In fact, we use techniques similar to those found here as well as Theorem \ref{thm:Main Theorem} in the sequel \cite{OedingTan} which investigates a connection between principal minors of symmetric matrices and the tangential variety to the Segre product of projective spaces and solves the set-theoretic version of a conjecture of Landsberg and Weyman \cite{LWtan}.

%%%%%%%%%%%%%%%%%  Outline %%%%%%%%%%%%%%%%%%%%%%

\subsection{Extended outline}
The rest of the paper is organized as follows. In Section~\ref{sec:intro:apps} we discuss applications of Theorem~\ref{thm:Main Theorem} to Statistics, Physics and Graph Theory.
In Section~\ref{sec:tensorreps} we recall basic notions concerning tensors, representations and $G$-varieties. We point out many standard facts from representation theory that we will use to study the ideal of $Z_{n}$ and the hyperdeterminantal module. In particular, we recall a method used by Landsberg and Manivel to study $G$-modules of polynomials via Schur modules.  We also show how to use weights and lowering operators to describe and identify Schur modules. We use these concepts in our proof of Lemma~\ref{lem:almost}.  Lemma~\ref{lem:almost} is the key to  Proposition~\ref{prop:no almost} which is crucial to our proof of Theorem~\ref{thm:Main Theorem}.

In Sections~\ref{sec:defs} through~\ref{sec:augment} we describe geometric aspects of the variety of principal minors of symmetric matrices and the zero set of the hyperdeterminantal module.
In Section~\ref{sec:defs} we set up notation and give a precise definition of the variety.  We also recall two useful facts; a symmetric matrix is determined up to the signs of its off-diagonal terms by its $1\times 1$ and $2 \times 2$ principal minors, and the dimension of $Z_{n}$ is $\binom{n+1}{2}$.  In Section~\ref{sec:structure} we describe the nested structure of $Z_{n}$.  In particular, in Proposition~\ref{prop:ZZ block} we show that $Z_{n}$ contains all possible Segre products of $Z_{p}$ and $Z_{q}$ where $p+q=n$.  We use this interpretation in Proposition~\ref{prop:no almost}.

In Section~\ref{sec:module} we study properties of the hyperdeterminantal module.  In particular, we point out that it has dimension $\binom{n}{3}5^{n-3}$.  In Proposition~\ref{prop:multiplicity} we show that it actually is an irreducible $\G$-module of polynomials that occurs with multiplicity $1$ in the space of degree $4$ homogeneous polynomials. This is a consequence of a more general fact about modules with structure similar to that of the hyperdeterminantal module, which we record in Lemma~\ref{lem:multiplicity:general}.  In Proposition~\ref{prop:one way} we record the fact (originally proved in \cite{HSt}) that the hyperdeterminantal module is in the ideal of $Z_{n}$. Then in Proposition~\ref{prop:unaugment} we generalize the idea to other varieties that have similar structure.

In Section~\ref{sec:augment} we extract a general property of the hyperdeterminantal module that we call augmentation. We explore properties of augmented modules via polarization of tensors, a technique from classical invariant theory used, for example, in the study of secant varieties.  Of particular interest is the Augmentation Lemma~\ref{lem:augment}, in which we give a geometric description of the zero set of a general augmented module.  We apply the Augmentation Lemma~\ref{lem:augment} to give a geometric characterization of the zero set of the hyperdeterminantal module in Lemma~\ref{lem:characterization}.  We use Lemma~\ref{lem:characterization} in the proof of Theorem~\ref{thm:Main Theorem}.  Proposition~\ref{prop:segre ideal} is another application of the Augmentation Lemma to polynomials that define Segre products of projective spaces. We use a slightly more complicated version of Proposition~\ref{prop:segre ideal} in the proof of Lemma~\ref{lem:almost}.

In Sections~\ref{sec:understanding} and~\ref{sec:theorem} we pull together all of the ideas from the previous parts to prove Theorem~\ref{thm:Main Theorem}. In particular, we show that any point in the zero set of the hyperdeterminantal module has a symmetric matrix that maps to it under the principal minor map. 

In Section~\ref{sec:understanding} we work to understand the case when all principal minors of a symmetric matrix agree with a given vector except possibly the determinant.
Of particular importance is Proposition~\ref{prop:no almost} which essentially says that for $n\geq 4$, if $z$ is a vector in the zero set of the hyperdeterminantal module, then a specific subset of the coordinates of $z$ determine the rest of its coordinates.

In order to prove Proposition~\ref{prop:no almost}, we use practically all of the tools from  representation theory that we have introduced and developed earlier in the paper.  With the aid of Proposition~\ref{prop:no almost}, we prove Theorem~\ref{thm:Main Theorem} in Section~\ref{sec:theorem}.

%%%%%%%%%%% ^^^^^^^^^ Introduction ^^^^^^^ %%%%%%%%%%%%

%%%%%%%%%%%%% Applications %%%%%%%%%
\section{Applications of Theorem~\ref{thm:Main Theorem}}\label{sec:intro:apps}
We conclude this introduction by describing how Theorem~\ref{thm:Main Theorem} answers questions in other areas via three examples; in Statistics and the study of negatively correlated random variables, in Physics and the study of determinantal point processes, and in Spectral Graph Theory  and the study of graph invariants.  \cite{HSc, HSt,GriffTsat2, borcea-2009-22, Holtz, Wagner, Mikkonen}

\subsection{Application to covariance of random variables}

Consider a non-singular real symmetric $n\times n$ matrix $A$.  The principal minors of $A$ can be interpreted as values of a function $\omega : \mathcal{P}(\{1,\dots,n\}) \rightarrow [0,\infty)$, where $\mathcal{P}$ is the power set. This function $\omega$, under various restrictions, is of interest to statisticians. In this setting, the off-diagonal entries of the matrix $A^{-1}$ are associated to covariances of random variables.  In D. Wagner's \cite{Wagner} asked the following: 

\begin{question} When is it possible to prescribe the principal minors of the matrix $A$ as well as the off-diagonal entries of $A^{-1}$?
\end{question}

In \cite[Theorem 6]{HSt} this question is answered using the hyperdeterminantal equations in degree 4, another set of degree 10 equations and the strict Hadamard-Fischer inequalities.

Our main result provides an answer to the first part of the question:

It is possible to prescribe the principal minors of a symmetric matrix if and only if the candidate principal minors satisfy all the relations given by the hyperdeterminantal module.

For the second part of the question we can give a partial answer. It is not hard to see that the off-diagonal entries of $A^{-1}$ are determined up to sign by the $0\times 0$, $1\times 1$ and $2 \times 2$ principal minors, and the rest of the principal minors further restrict the freedom in the choices of signs.

Another useful fact is if $A$ is invertible then 
\[
A^{-1} = \frac{\adj(A)}{\det(A)}
,\]
where $\adj(A)_{i,j} =  ( (-1)^{i+j} det(A^{j}_{i}))$ is the adjugate matrix.
So up to scale, the vector of principal minors of $A^{-1}$ is the vector of principal minors of $A$ in reverse order.  Therefore the determinant, $n-1 \times n-1$ and $n-2 \times n-2$ principal minors of $A$ determine the off diagonal entries of $A^{-1}$ up to $\binom{n}{2}$ choices in combinations of signs, and the rest of the principal minors further restrict the choices of combinations of signs.

\subsection{Application to determinantal point processes}
Determinantal point processes were introduced by Macchi in 1975, and subsequently have received significant attention in many areas. A non zero point $p_{S} \in \CC^{2^{n}}$ is called determinantal if there is an integer $m$ and an $(n+m)\times(n+m)$ matrix $K$ such that for $S\subset \{1,2,\dots,n\}$
\[
p_{S} = \det_{S\cup \{n+1,\dots,n+m\}}(K)
.\]
Borodin and Rains were able to completely classify all such points for the case $n=4$  (Theorem 4.6 \cite{BorodinRains}) by giving a nice geometric characterization.  Lin and Sturmfels \cite{LinSturmfels} studied the geometric and algebraic properties of the algebraic variety of determinantal points and independently arrived at the same result as Borodin and Rains. Moreover, Lin and Sturmfels gave a complete proof of the claim of \cite{BorodinRains} that the ideal of the variety is generated in degree 12 by 718 polynomials.

Consider the case where we impose the restrictions that the matrix $K$ to be symmetric and the integer $m=0$, and call these restricted determinantal points \emph{symmetric determinantal points}. 

\textbf{Restatement:}
The variety of all symmetric determinantal points is cut out set-theoretically by the hyperdeterminantal module.

This restatement is useful because it provides a complete list of necessary and sufficient conditions for determining which symmetric determinantal points can possibly exist.

\subsection{Application to spectral graph theory}
A standard construction in graph theory is the following.
To a weighted directed graph $\Gamma$ one can assign an adjacency matrix $\Delta(\Gamma)$.

The eigenvalues of $\Delta(\Gamma)$ are invariants of the graph. The first example is with the standard graph Laplacian.  Kirchoff's well-known Matrix--Tree theorem states that any $(n-1)\times( n-1)$ principal minor of $\Delta(\Gamma)$ counts the number of spanning trees of $\Gamma$.

There are many generalizations of the Matrix--Tree Theorem, such as the Matrix--Forest Theorem which states that $\Delta(\Gamma) ^{S}_{S}$, the principal minor of the graph Laplacian formed by omitting rows and columns indexed by the set $S \subset \{1,\dots,n\}$,  computes the number of spanning forests of $\Gamma$ rooted at vertices indexed by $S$.

The principal minors of the graph Laplacian are graph invariants. The relations among principal minors are then also relations among graph invariants.  Relations among graph invariants are central in the study of the theory of unlabeled graphs. In fact, Mikkonen holds that ``the most important problem in graph theory of unlabeled graphs is the problem of determining graphic values of arbitrary sets of graph invariants,''  (see \cite{Mikkonen} p.~1).
  
Theorem~\ref{thm:Main Theorem} gives relations among the graph invariants that come from principal minors, and in particular, since a graph can be reconstructed from a symmetric matrix, Theorem~\ref{thm:Main Theorem} implies the following:

\textbf{Restatement:}  There exists an undirected weighted graph $\Gamma$ with invariants $[v]\in \PP^{2^{n}-1}$ specified by the principal minors of a symmetric matrix $\Delta(\Gamma)$ if and only if $[v]$ is a zero of all the polynomials in the hyperdeterminantal module.

%
%%%%%%%%%%%%%%  ^^^^^^^^^  Applications  ^^^^^^^^^     %%%%%%%%%%%%%%%%%%%

%%%%%%%%%%% Background %%%%%%%%%%%%%

\section{Background on $G$-varieties in spaces of tensors and their ideals as $G$-modules}\label{sec:tensorreps}
An $n\times n$ matrix has $2^{n}$ principal minors (determinants of submatrices centered on the main diagonal), so vectors of principal minors may be considered  in the space $\CC^{2^{n}}$.
However, the natural ambient space for vectors of principal minors from the point of view of symmetry (Theorem \ref{thm:symmetry}) is the $n$-fold tensor product $\CC^{2}\otdot \CC^{2}$.  With this setting in mind, in this section we study tensor products of several vector spaces, natural group actions on tensors, representation theory for tensor products, and classical subvarieties in spaces of tensors.

For the sake of the reader not familiar with representation theory, we have chosen to include many definitions and basic concepts that we might have skipped otherwise. For more background, one may consult \cite{FultonHarris, LandsbergTensorBook,GoodWall, Weyl, Harris, CLO}. 

If $V$ is a vector space and $G\subset \GL(V)$, a variety $X\subset \PP V$ is said to be a \emph{$G$-variety} or \emph{$G$-invariant} if it is preserved by the action of $G$, specifically $g.x\in X$ for every $x\in X$ and $g\in G$.  In this article our vector spaces are always assumed to be finite dimensional.
Our study fits into the more general context of arbitrary $G$-varieties for a linearly reductive group $G$, and we sometimes allude to this setting, but for the sake of efficiency and clarity we often present the necessary representation-theoretic concepts only in the case of tensors.  The expert reader might try to envision the basic techniques we use in their more general context.

\subsection{Examples of classical $G$-varieties in spaces of tensors}
Let $V_{1},\dots,V_{n}$ be complex vector spaces and let $V_{1}\otdot V_{n}$ denote their tensor product.
The following are two classic examples of $G$-varieties in the space of tensors $\PP (V_{1}\otdot V_{n})$ which happen to show up in our study of the variety of principal minors of symmetric matrices.  These definitions can be found in many texts on algebraic geometry such as \cite{Harris}.

The space of all rank-one tensors (also called decomposable tensors) is the \emph{Segre variety}, defined by the embedding,
\begin{align*} Seg: \PP V_{1} \tdt \PP V_{n} &\longrightarrow \PP \left(V_{1}\otdot V_{n}\right) \\
 ([v_{1}],\dots,[v_{n}]) &\longmapsto  [v_{1}\otdot v_{n}]
.\end{align*}
$Seg\left(\PP V_{1}\tdt \PP V_{n} \right)$ is a $G$-variety for $G = \GL(V_{1})\tdt \GL(V_{n})$, moreover it is \emph{homogeneous} (the $G$-orbit of a single point) since $Seg\left(\PP V_{1}\tdt \PP V_{n} \right) = G.[v_{1}\otdot v_{n}]$. If $X_{1}\subset \PP V_{1},\dots, X_{n} \subset \PP V_{n}$ are varieties, let $Seg\left(X_{1}\tdt X_{n} \right)$ denote their Segre product.

The $r^{th}$ \emph{secant variety} to a variety $X \subset\PP V$, denoted $\sigma_{r}(X)$, is the Zariski closure of all embedded secant $\PP^{r-1}$'s to X, \emph{i.e.},
\[
\sigma_{r}(X) = \overline{\bigcup_{x_{1},\dots,x_{r} \in X} \PP( span\{x_{1},\dots,x_{r}\})}\subset \PP V
.\]
Secant varieties inherit the symmetry of the underlying variety.  In particular, 
\[\sigma_{r}\left(Seg\left(\PP V_{1}\tdt \PP V_{n} \right)\right)\] is a $G$-variety for $G = \GL(V_{1})\tdt \GL(V_{n})$.  However, homogeneity is not preserved in general.

\subsection{The variety of principal minors of symmetric matrices}\label{sec:defs}
Let $I = [i_{1},  \dots  i_{n}]$ be a binary multi-index, with $i_k \in \{0,1\}$ for $k = 1,  \dots , n$, and let $| I | = \sum_{k=1}^{n}i_{k}$. 
A natural basis of $(\CC^{2})^{\otimes n}$ is the set of tensors
 $X^{I} := x_{1}^{i_{1}}\otimes x_{2}^{i_{2}}\otimes  \dots \otimes x_{n}^{i_{n}}$ for all length $n$ binary indices $I$.
We use this basis to introduce coordinates; if 
$P= [C_{I}X^{I}] \in \PP (\CC^{2})^{\otimes n}$, the coefficients 
$C_{I}$ are the homogeneous coordinates of the point $P$. (Note we use the summation convention that the implied summation is over the index $I$ which appears as a superscript and a subscript.)

Let $S^{2}\CC^{n}$ denote the space of symmetric $n\times n$ matrices. If $A \in S^{2}\CC^{n}$, then let $\Delta_I(A)$ denote the principal minor of $A$ formed by taking the determinant of the principal submatrix of $A$ indexed by $I$ in the sense that the submatrix of $A$ is formed by including the $k^{th}$ row and column of $A$ whenever $i_{k} = 1$ and striking the $k^{th}$ row and column whenever $i_{k}=0$.
%If one includes the $0\times 0$ minor, there are $2^{n}$ principal minors, therefore, a natural home for vectors of principal minors is $\CC^{2^{n}}$.
%However, in light of Theorem~\ref{thm:symmetry}, it is practical to view $\CC^{2^{n}}$ with the extra structure as a space of tensors $(\CC^{2})^{\otimes n}$ as the symmetry of principal minors is more evident here.

The projective variety of principal minors of $n\times n$ symmetric matrices, $Z_{n}$, is defined by the following rational map,
\begin{align*}
 \varphi   : \PP(S^{2} \CC^n \oplus  \CC)  & \dashrightarrow  \PP (\CC^{2})^{\otimes n} \\
 [A,t]  & \longmapsto  \left[t^{n-|I|}\Delta_{I}(A)\; X^I \right].
\end{align*}
The map $\varphi$ is defined on the open set where $t\neq0$. Moreover, $\varphi$ is homogeneous of degree $n$, so it is a well-defined rational map on projective space.  The $1\times 1$ principal minors of a matrix $A$ are the diagonal entries of $A =(a_{i,j})$, and if $A$ is a symmetric matrix, the $1\times 1$ and $2 \times 2$ principal minors determine the off-diagonal entries of $A$ up to sign in light of the equation
\[
a_{i,i}a_{j,j}-a_{i,j}^{2} = \Delta_{[0,\dots,0,1,0,\dots,0,1,\dots,0]}(A)
,\]
where the $1$'s in the index $[0,\dots,0,1,0,\dots,0,1,\dots,0]$ occur in positions $i$ and $j$.
So $\varphi$ is generically finite-to-one and $Z_{n}$ is a $\binom{n+1}{2}$-dimensional variety.  The affine map (on the set $\{t=1\}$) defines a closed subset of $\CC^{2^{n}}$, \cite{HSt}.

\subsection{Ideals of $G$-varieties in spaces of tensors}
Let $V$ be a finite dimensional vector space over $\CC$.
Let  $V^{*}$ denote the dual vector space of linear maps $V \rightarrow \CC$. 
Let $S^{d}V^{*}$ denote the space of homogeneous degree $d$ polynomials on $V$, and let $\Sym(V^{*}) = \bigoplus_{d}S^{d}V^{*}$ denote the polynomial ring.

If $X\subset \PP V$ is a projective algebraic variety, let $\I(X) \subset \Sym(V^{*})$ denote the ideal of polynomials vanishing on $X$, and let $\widehat{X}\subset V$ denote the cone over $X$. If $M$ is a set of polynomials, let $\V(M)$ denote its zero set.  Often algebraic varieties are given via an explicit parameterization by a rational map, but the vanishing ideal may be unknown.  A basic question in algebraic geometry is to find generators for the ideal of a given variety.  Though there are many known theoretical techniques, this remains a difficult practical problem.

\textbf{Fact:} $X$ is a $G$-variety if and only if $\I(X)$ is a $G$-module. This fact, which comes directly from the definitions, is a key observation because it allows us to use the representation theory of $G$-modules to study $\I(X)$.

By definition, all projective varieties are preserved by the action of $\CC\setminus\{0\}$ by rescaling. It is well know that this action induces a grading by degree on the ideal, $\I(X) = \bigoplus_{d}\I_{d}(X)$ where $\I_{d}X:=S^{d}(V^*)\cap \I(X)$.
In parallel, when a larger, linearly reductive group $G$ acts on $X$, we get a finer decomposition of each  module $\I_{d}(X)$ into a direct sum of irreducible $G$-modules.  The irreducible modules in $\I_{d}(X)$ are a subset of those in $S^{d}V^{*}$.  This simple observation leads to a useful ideal membership test, which is developed and discussed in \cite{LM04, LandsbergTensorBook}.

The group $\GL(V_{1})\tdt \GL(V_{n})$ acts on $V_{1}\otdot V_{n}$ by change of coordinates in each factor.  
%
%If $V_{i}\simeq V$ for every $i$ we can consider the induced action of $\GL(V)$ on the tensor product, 
%\begin{eqnarray*} 
%\GL(V) \times V\otdot V &\longrightarrow & V\otdot V \\
%(g,x_{1}\otdot x_{d}) & \longmapsto & (g.x_{1}) \otimes (g.x_{2}) \otdot (g.x_{d})
%,\end{eqnarray*}
%where $g.x_{i}$ is the usual action of $\GL(V)$ on $V$ and we extend the action via linearity.
%
When $V_{i}$ are all isomorphic, there is also a natural action of the symmetric group $\mathfrak{S}_{n}$ on $V_{1}\otdot V_{n}$   by permuting the factors.  
%More specifically, the left action is given (on a basis) by
%\begin{eqnarray*} 
%\mathfrak{S}_{d} \times V\otdot V &\longrightarrow & V\otdot V \\
%(\sigma,(x_{1}\otdot x_{d})) & \longmapsto & x_{\sigma^{-1}(1)} \otdot x_{\sigma^{-1}(d)} 
%.\end{eqnarray*}
With this convention one may define a left action of the semi-direct product $\GL(V) \ltimes \mathfrak{S}_{n}$ on $V^{\otimes n}$.

If $V$ is a vector space and $G\subset GL(V)$, we say that is a \emph{$G$-module} or \emph{a representation of $G$}, if it is preserved by the action of $G\subset GL(V)$.
A $G$-module said to be  \emph{irreducible} if it has no non-trivial $G$-invariant subspaces.  

The general linear group $\GL(V)$ has well understood representation theory.  In particular \cite[Proposition 15.47]{FultonHarris} says that every $\GL(V)$-module is isomorphic to a Schur module of the form $S_{\pi}V $, where $\pi$ is a partition of an integer $d$.  We refer the reader to \cite{FultonHarris,LandsbergTensorBook} for general background on Schur modules.

Two common representations (in this language) are the space of symmetric tensors $S^{d}V = S_{(d)} V$ and the space of skew-symmetric tensors $\bigwedge^{d}V = S_{(1^{d})}V$, where $1^{d}$ denotes the partition $(1,\dots,1)$ with $1$ repeated $d$ times.  

We will be interested in representations of $\SL(V)$.  In light of the isomorphism $\SL(V) \cong \GL(V)/Z(\GL(V))$, where the center $Z(\GL(V)) = \CC\setminus\{0\}$ is isomorphic to scalar multiples of the identity, the representation theory of $\GL(V)$ is essentially the same at that of $\SL(V)$.  Specifically, if $V$ is $m$-dimensional, two representations  $S_{\pi}V$ and $S_{\lambda}V$ of $\GL(V)$ are isomorphic as $\SL(V)$ modules if $\pi = \lambda + k^{m}$, some $k\in \ZZ$, where $k^{m}$ is the partition $(k,\dots,k)$ with $k$ repeated $m$ times.  
%In particular, as $\SL(\CC^{m})$-modules, for each $k\in \ZZ$, the representation $S_{k^{m}}\CC^{m}$ is isomorphic to the representation $\CC$ of scalar multiples of the identity.
However, since we care about how the modules we are studying are embedded in the space of polynomials, we will not reduce partitions via this equivalence.

% There is a natural action of $\GL(V)$ on $V^{*}$ defined by  $g.\omega(x) := \omega (g^{-1}. x)$ for every $x\in V$, $g \in \GL(V)$ and $\omega\in V^{*}$.
%
%The irreducible representations of $GL(V_{1})\tdt GL(V_{n})$ are all isomorphic to tensor products of Schur modules, \ie they are of the form
%\[
%S_{\pi_{1}}V_{1}\otdot S_{\pi_{n}}V_{n}
%,\]
%where $\pi_{i}$ $1\leq i\leq n$ are partitions.

We are interested in the case when $X \subset \PP (V_{1}\otdot V_{n})$ is a variety in a space of tensors, and $X$ is invariant under the action of $G= \GL(V_{1}) \tdt  \GL(V_{n})$. 
To study $\I_d(X)$ as a $G$-module, we need to understand how to decompose the space of homogeneous degree $d$ polynomials $S^d(V_{1}^*\otdot V_{n}^*)$ into a direct sum of irreducible $G$-modules.  This is a standard computation in representation theory, which has been made explicit for example in \cite{LM04}.

\begin{prop}[Landsberg--Manivel \cite{LM04} Proposition 4.1]\label{LMdecomp}
 Let $V_1, \dots , V_n$ be vector spaces and let $G= GL(V_1)\tdt GL(V_n)$. Then the following decomposition as a direct sum of irreducible $G$-modules holds:
\[S^d(V_{1}^{*}\otimes\dots\otimes V_{n}^{*}) = \bigoplus _{|\pi_1|=\dots=|\pi_n|=d} ([\pi_1]\otimes\dots\otimes[\pi_n])^{\mathfrak{S}_d} \otimes S_{\pi_1}V_{1}^{*} \otimes\dots\otimes S_{\pi_n} V_{n}^{*}
\]
where $[\pi_{i}]$ are representations of the symmetric group $\mathfrak{S}_{d}$ indexed by partitions $\pi_{i}$ of $d$, and $([\pi_1]\otimes\dots\otimes[\pi_n])^{\mathfrak{S}_d}$ denotes the space of $\mathfrak{S}_d$-invariants (\ie, instances of the trivial representation) in the tensor product.
\end{prop}

%\begin{proof}[Proof from \cite{LM04}]
% Schur--Weyl duality is the assertion that the following map is an isomorphism of $GL(V)$-modules
%\[\bigoplus_{|\pi|=d}[\pi]\otimes S_{\pi}V\longrightarrow V^{\otimes d}.
%\]
% Apply Schur--Weyl duality separately to each of $V_1, \dots , V_n$, take the tensor product of the corresponding isomorphisms, and compare with Schur duality for $V_1\otimes \dots \otimes V_{n}$.
%\end{proof}

When the vector spaces $V_{i}^{*}$, are all isomorphic to the same vector space $V^{*}$,  Proposition~\ref{LMdecomp} specializes to give the following decomposition formula (as $GL(V)\tdt GL(V)$-modules) also found in \cite{LM04}:
\begin{equation}\label{eq:decomp}
S^d(V^{*}\otimes\dots\otimes V^{*}) = \bigoplus _{|\pi_1|=\dots=|\pi_n|=d}  \left(S_{\pi_1}V^{*} \otimes\dots\otimes S_{\pi_n} V^{*} \right)^{\oplus N_{\pi_{1},\dots,\pi_{k}}}
,\end{equation}
where the multiplicity $N_{\pi_{1},\dots,\pi_{k}}$ can be computed via characters.  The modules $(S_{\pi_1}V^{*} \otimes\dots\otimes S_{\pi_n} V^{*})^{\oplus N_{\pi_{1},\dots,\pi_{k}}}
$ are called \emph{isotypic components}.
 
The irreducible $\SL(V)^{\times n}\ltimes \mathfrak{S}_{n}$-modules are constructed by taking an irreducible $\SL(2)^{\times n}$ module $S_{\pi_{1}}V \otdot S_{\pi_{n}} V$ and summing over all permutations in $\mathfrak{S}_{n}$ that yield non-redundant modules. When the vector space is understood, we denote this compactly as
\begin{equation*}%\label{modnotation}
S_{\pi_{1}}S_{\pi_{2}}\dots S_{\pi_{n}} := \sum_{\sigma \in \mathfrak{S}_{n}} S_{\pi_{\sigma(1)}}V^{*} \otdot S_{\pi_{\sigma(n)}} V^{*}
\end{equation*}

The decomposition formula \eqref{eq:decomp} is essential for understanding the structure of the ideals of $G$-varieties. There is an implementation of \eqref{eq:decomp} in the computer program LiE, and we wrote an implementation in Maple.

The combinatorial description of Schur modules in terms of collections of partitions can be used to construct polynomials in spaces of tensors.  We refer the reader to \cite{LM04, LandsbergTensorBook} for a complete explanation. A copy of our implementation of these algorithms may be obtained by contacting the author.

\subsection{Weights, raising operators, and highest weight vectors}\label{sec:WeightSpaces} 
The notions of \emph{weights}, \emph{weight vectors}, \emph{highest weight vectors}, and \emph{raising/lowering operators} are well-known practical tools for studying representations and polynomials in spaces of tensors.
Here we recall definitions and concepts that can be found in standard textbooks on representation theory in order to define the terms we use in this paper and to explain our use of these representation-theoretic tools.
Many of the concepts in this section are practical re-interpretations of concepts in the previous section.

%We specialize to the case when $V_{i}\simeq \CC^{2}$ because this is notationally simpler, and it is all we need for this work. For the rest of the paper we also focus almost exclusively on representations for products of $\SL(V)$ instead of $\GL(V)$.  The representation theory is essentially the same up to a scalar multiple of the determinant. The main difference is that while the modules $S_{\pi^{1},\pi^{2}}\CC^{2}$ and $S_{\lambda^{1},\lambda^{2}}\CC^{2}$ are isomorphic $\SL(2)$ modules if and only if $\pi^{2}-\pi^{1} = \lambda^{2}-\lambda^{1}$, whereas they are isomorphic $\GL(2)$ modules if and only if $\pi = \lambda$.  While the $\SL(2)$ modules are all determined up to isomorphism by a single integer, we prefer to continue to use the partitions of $2$ parts to index the irreducible modules because, in addition to avoiding ambiguity, this helps to understand how a given irreducible module of polynomials is embedded in the space of all polynomials.

Choose a basis $\{x_{i}^{0},x_{i}^{1}\}$ for each $V_{i}$ and assign the integer \emph{weight} $-1$ to $x_{i}^{0}$ and the weight $+1$ to $x_{i}^{1}$.
Weights of tensors in the algebra $(V_{1}\otdot V_{n})^{\otimes}$ are length-$n$ integer vectors defined first on monomials then extended by linearity. Specifically,
\[
(x_{1}^{0})^{\otimes p_{1}} \otimes (x_{1}^{1})^{\otimes q_{1}} \otimes (x_{2}^{0})^{\otimes p_{2}} \otimes (x_{2}^{1})^{\otimes q_{2}} \otdot (x_{n}^{0})^{\otimes p_{n}} \otimes (x_{n}^{1})^{\otimes q_{n}}
\]
has weight
\[ (q_{1}-p_{1},q_{2}-p_{2},\dots,q_{n}-p_{n}).\]
A tensor is called a \emph{weight vector} if all of its monomials have the same weight, and this is the only time it makes sense to assign a weight to a tensor.
%Referee's suggestion
This is the standard assignment of weights for the connected component containing the identity in $\G$, and is also known as grading by multi-degree.
  
%We choose the lexicographic order on weights. Then for a given subset of weights, the \emph{highest weight} (respectively \emph{lowest weight}), is the weight occurring soonest (respectively latest) in the lexicographic order. As a caution, we point out that with this convention, smaller, or numbers that are more negative actually indicate a higher weight.
%In order, the weights of $S_{(\pi^{1},\pi^{2})}\CC^{2}$ are precisely the integers $\pi^{2}-\pi^{1}, \pi^{2}-\pi^{1}+2, \dots \pi^{1}-\pi^{2}$.

The Lie algebra $\mathfrak{g}$ associated to the Lie group $G$ acts on $G$-modules by derivation. We assume that $G$ is a linearly reductive connected algebraic group, of which $\SL(V)^{\times n}$ is an example.
 An essential fact we will use is that  $M$ is a $G$-module if and only if $M$ is a $\mathfrak{g}$-module.  We have a decomposition $\mathfrak{g}= \mathfrak{g_{-}} \oplus \mathfrak{g_{0}} \oplus \mathfrak{g_{+}}$ into the lowering operators, the Cartan (Abelian) subalgebra and the raising operators.  

The Lie algebra of $SL(2)$ is $\mathfrak{sl}(2)$,  the algebra of traceless $2\times 2$ matrices acting as derivations.  The raising (respectively lowering) operators can be thought of as upper (respectively lower) triangular matrices when 
For example the lowering operator in $\mathfrak{sl}_{2}$ acts on $V = \{x_{0},x_{1}\}$ by sending $x_{0}$ to a scalar multiple of $x_{1}$ and sending $x_{1}$ to $0$.  

The Lie algebra of $SL(2)^{\times n}$ is $\mathfrak{sl}_{2}^{\oplus n}$ where each $\mathfrak{sl}_{2}$ acts on a single factor of the tensor product $V_{1}\otdot V_{n}$.
This action is extended to $S^{d}(V_{1}\otdot V_{n})$ by noting that the differential operators obey the Leibnitz rule.  The raising (lowering) operators fix the degree of a polynomial.

A weight vector in a $G$-module is called a \emph{highest weight vector} (respectively \emph{lowest weight vector}) if it is in the kernel of all of the raising (respectively lowering) operators.
Consider the irreducible module $S_{\pi_{1}}V_{1}^{}\otdot S_{\pi_{n}}V_{n}$ with each $\pi_{i}$ a partition of $d$.
Since $V_{i}\simeq \CC^{2}$ for every $1\leq i\leq n$, each $\pi_{i}$ is of the form $(\pi_{i}^{1},\pi_{i}^{2})$ with $\pi_{i}^{1}+\pi_{i}^{2} = d$.  A highest weight vector in $S_{\pi_{1}}V_{1}\otdot S_{\pi_{n}}V_{n}$ has weight $(\pi_{1}^{2}-\pi_{1}^{1},\pi_{2}^{2}-\pi_{2}^{1},\dots,\pi_{n}^{2}-\pi_{n}^{1})$.
If $w$ is the weight of a nonzero vector in $S_{\pi_{1}}V_{1}\otdot S_{\pi_{n}}V_{n}$ then $-w$ is also the weight of a nonzero vector, and if $w$ is the weight of a highest weight vector in a module then $-w$ is the weight of a lowest weight vector.

\textbf{Fact:} Assume $G$ is a linearly reductive connected algebraic group. Each finite dimensional irreducible $G$-module is the span of the $G$-orbit of a highest (or lowest) weight vector. 

\begin{remark}\label{rmk:transition}
If $T$ is a nonzero homogeneous polynomial on $V_{1}\otdot V_{n}$, and $T$ is a highest (or lowest) weight vector, then the degree $d$ and weight $(w_{1},w_{2},\dots,w_{n})$ of $T$ is sufficient information to determine (up to isomorphism) a module of the form $S_{\pi_{1}}V_{1}\otdot S_{\pi_{n}}V_{n}$ in which it occurs. We say that we know in which isotypic component the module lives. Specifically, we have $d = \pi_{i}^{1}+\pi_{i}^{2}$ and $w_{i} = \pi_{i}^{2}-\pi_{i}^{1}$, so $\pi_{i} = \frac{1}{2}(d-w_{i}, d+w_{i})$.

In general, the degree and weight of a highest weight polynomial will not be sufficient to find how the module $S_{\pi_{1}}V_{1}^{*}\otdot S_{\pi_{n}}V_{n}^{*}$ is embedded in $S^{d}(V_{1}^{*}\otdot V_{n}^{*})$ (\ie how it is embedded in the isotypic component).  On the other hand, if the found module occurs with multiplicity one in $S^{d}(V_{1}^{*}\otdot V_{n}^{*})$, then the degree and weight of a highest weight vector is sufficient information to identify the module.

\begin{example}\label{rmk:HDweight}
The hyperdeterminant of format $2 \times 2 \times 2$ is invariant under the action of $\SL(2)\times \SL(2) \times \SL(2)$, therefore it must have weight $(0,0,0)$.  This, together with the knowledge that it is a degree $4$ polynomial annihilated by each raising operator immediately tells us that it must be in the module $S_{(2,2)}\CC^{2}\otimes S_{(2,2)}\CC^{2}\otimes S_{(2,2)}\CC^{2}$ which occurs with multiplicity one in $S^{4}(\CC^{2}\otimes \CC^{2} \otimes \CC^{2})$.  Moreover, one can write the $2 \times 2 \times 2$ hyperdeterminant on the variables $X^{[i_{1},i_{2},i_{3},0,\dots,0]}$.  The weight of this polynomial is $(0,0,0,-4,\dots,-4)$ and it is a highest weight vector, therefore the span of its $\G$-orbit is the hyperdeterminantal module,
\[
HD := S_{(2,2)}S_{(2,2)}S_{(2,2)}S_{(4)}\dots S_{(4)}
.\]
\end{example}
\end{remark}

\subsection{An algorithm to produce a $G$-module from a polynomial}\label{sec:lowering}
Suppose we can write down a polynomial $h$ in some (unknown) $G$-module $M$. (Again we are assuming that our modules are finite dimensional and the group $G$ is a linearly reductive connected algebraic group, and specifically thinking of the example $G=\SL(V)^{\times n}$.) Since $M$ is a $G$-module, it is also a $\mathfrak{g}$-module, where $\mathfrak{g}$ is the Lie algebra associated to the Lie group $G$. The following algorithm is a standard idea in representation theory and can be used to find more polynomials in $M$, and in fact we will find submodules of $M$.
In particular, this procedure is essential in the proof of Lemma~\ref{lem:almost} below.

By successively applying lowering operators, we will determine the lowest weight space in which a summand of $h$ can live.  The lowest weight vector that we construct will generate a submodule of $M$.

\noindent
\textbf{Input:} $h\in M$.

\textbf{Step $0$.} Choose an ordered basis of lowering operators $\mathfrak{g}_{-} = \{\alpha_{1}, \dots ,\alpha_{n}\}$.

\textbf{Step $1$.} Find the largest integer $k_{1} \geq 0$ so that  $\alpha_{1}^{k_{1}}.h \neq 0$, and let $h^{(1)} = \alpha_{1}^{k_{1}}.h $.

\textbf{Step $2$.} Find the largest integer $k_{2} \geq 0$ so that  $\alpha_{2}^{k_{2}}.h^{(1)} \neq 0$, and let $h^{(2)} = \alpha_{2}^{k_{2}}.h^{(1)} $.

\textbf{Step $n$.} Find the largest integer $k_{n} \geq 0$ so that $\alpha_{n}^{k_{n}}.h^{(n-1)} \neq 0$, and let $h^{(n)} = \alpha_{n}^{k_{n}}.h^{(n-1)} $.

\noindent
\textbf{Output:} The vector $h^{(n)}$ is a lowest weight vector in $M$ and $ span\{G.h^{(n)}\}$ is a submodule of $M$ containing $h^{(n)}$. 

Note, in the case $\mathfrak{g} = \mathfrak{sl}(2)^{\oplus n}$, the natural ordered basis of  $(\mathfrak{sl}(2)^{\oplus n})_{-}$ is $ \{\alpha_{1}, \dots ,\alpha_{n}\}$, where $\alpha_{i}$ is the lowering operator acting on the $V_{i}^{*}$ factor in $S^{d}(V_{1}^{*}\otdot V_{n}^{*})$.
%\begin{remark}\label{rmk:lowering}
%Suppose $h$ is a polynomial in a (not necessarily irreducible) $\G$-module $M$ of polynomials.  Then $M$ is also a $\mathfrak{sl}_{2}^{\oplus n}$-module and, in particular, applying a lowering operator to $h$ produces a new vector in $M$ of lower weight.

%Successive applications of lowering operators to $h$ will eventually find a nonzero lowest weight vector which is necessarily in $M$. This lowest nonzero weight vector will generate an irreducible submodule of $M$.
%\end{remark}
\begin{remark}\label{rmk:weightbasis}
In the case that $M$ is irreducible, by the same procedure of applying lowering operators to (this time) a highest weight vector $h$, we can construct a \emph{weight basis} $M$, namely a basis of $M$ consisting of weight vectors in $M$ of every possible weight.
\end{remark}

%
%Any $\SL(2)$-module is also a module for the Lie algebra  $\mathfrak{sl}_{2}$ of traceless $2 \times 2$ matrices acting as differential operators. 

%The Lie algebra associated to $\SL(2)^{\times n}$ is $\mathfrak{sl}_{2}^{\oplus n}$ where each $\mathfrak{sl}_{2}$ acts on a single factor of the tensor product $V_{1}\otdot V_{n}$.  This action is extended to $S^{d}(V_{1}\otdot V_{n})$ by noting that the differential operators obey the Leibnitz rule.  The raising (lowering) operators fix the degree of a polynomial.
%%
%Notice that each raising (respectively lowering) operator changes the weight of a vector adding $-2$ (respectively adding $2$) to a single entry of the weight.

%%%%%%%%%%% ^^^^^^^^^  Background ^^^^^^^^^  %%%%%%%%%%%%

%%%%%%%%%%%%% Geometry %%%%%%%%%%%%%%%

\section{The nested structure of $Z_{n}$ via Segre products}\label{sec:structure}

\begin{prop}\label{prop:subvar}
The variety $ Seg(Z_{(n-1)} \times \PP V_{n})$ is a subvariety of $Z_{n}$. In particular, any point of $Seg(Z_{(n-1)} \times \PP V_{n})$ is, after a possible change of coordinates, the principal minors of an $(n-1)\times (n-1)$ block of an $n\times n$ matrix.
\end{prop}
\begin{proof}
We prove the second statement first. Let $[\eta\otimes v]$ be a point in $Seg(Z_{(n-1)} \times \PP V_{n})$.
Then change coordinates in $V_{n}$ to send $[\eta\otimes v]$ to $[\eta\otimes x_{n}^{0}]$. 
Now $[\eta\otimes x_{n}^{0}]$ is in $Seg(Z_{(n-1)} \times \PP\{x_{n}^{0}\})$ which is the image under $\phi$ of matrices of the form
\[
\left[
\left(\begin{array}{cc}
P & 0 \\ 0 & 0
\end{array}\right)
,t\right], 
\]
where $P$ is a symmetric $(n-1)\times(n-1)$ sub-matrix of an $n\times n$ symmetric matrix.

The first statement then follows immediately from the $\G$-invariance of $Z_{n}$ and the fact that the $\SL(V_{n})$ orbit of $Seg(Z_{(n-1)} \times \PP\{x_{n}^{0}\})$ is $Seg(Z_{(n-1)} \times \PP V_{n})$.
 \end{proof}

In fact, Proposition~\ref{prop:subvar} generalizes as follows.
\begin{prop}\label{prop:ZZ block} Let $p+q = n$ and $Z_{p}\subset \PP \left(V_{1}\otdot V_{p} \right)$ and $Z_{q}\subset \PP \left(V_{p+1}\otdot V_{n}\right)$. Then $Seg(Z_p \times Z_q)$ is a subvariety of $ Z_{n}$.

Let $U_{0}= \{[z] \in \PP(V_{1}\otdot V_{n})\mid z = z_{I}X^{I}\in V_{1}\otdot V_{n}, z_{[0, \dots ,0]} \neq 0 \}$. Then $\varphi([A,t]) \in Seg(Z_{p}\times Z_{q}) \cap U_{0}$,  if and only if $A$ is of the form
\[\left( \begin{array}{cc} P & 0 \\ 0 & Q \end{array} \right)
,\] where $P\in S^{2}\CC^{p}$ and $Q\in S^{2}\CC ^{q}$. 
\end{prop}

%Idea of proof: 
%Let $U, U_{1},U_{2}$ be vector spaces such that $U=U_{1} \oplus U_{2}$. Then  $S^{2}U_{1}\oplus S^{2}U_{2} \subset S^{2}U$.  How does this subset map into the variety?  We claim that we get $\varphi^n(S^{2}U_{1}\oplus S^{2}U_{2} \oplus \CC) = Seg( \varphi^p( S^{2}U_{1} \oplus \CC) \times \varphi^q( S^{2}U_{2} \oplus \CC))$, where 

\begin{proof}
Let $\varphi^{i}$ denote the principal minor map on $i\times i$ matrices and let $J$ and $K$ be (respectively) multi-indices of length $p$ and $q$.
 Let $[x\otimes y] \in Seg(Z_p \times Z_q)$ be such that $[x]=\varphi^p([P,r]) = [r^{p-|J|}\Delta_{J}(P) X^{J}]$ and $[y]=\varphi^q([Q,s]) = [s^{q-|K|}\Delta_{K}(Q) X^{K}]$, with $P\in S^{2}\CC^{p}$ and $Q\in S^{2}\CC^{q}$. 
 
Notice that if $r=0$, then $[x] = [0,\dots,0,det(P)] \in Seg(\PP V_{1}\tdt \PP V_{p})$, and similarly if $s=0$, then $[y] = [0,\dots,0,det(Q)] \in Seg(\PP V_{p+1}\tdt \PP V_{p+q})$. So the cases that $r=0$ or $s=0$ are covered by iterations of Proposition~\ref{prop:subvar}.

Now assume $r\neq 0,s\neq 0$ so we can set $r=s=1$. Consider a blocked matrix of the form
\begin{equation}\label{block matrix} A =
\left(\begin{array}{cc}
P & 0 \\ 0 & Q 
\end{array}\right)
,
\end{equation} 
where $P\in S^{2}\CC^{p}$ and $Q\in S^{2}\CC^{q}$.
We claim that $\varphi^{p+q}([A,1] ) = [x\otimes y]$.
The determinant of a block diagonal matrix is the product of the determinants of the blocks, and principal submatrices of block diagonal matrices are still block diagonal, so
\[ \varphi^{n}([A,1] ) = \left[ \Delta_{J}(P) \Delta_{K}(Q)\; X^{J,K} \right]  .\]
where $X^{J,K} = X^J\otimes X^K$. But we can reorder the terms in the product to find 
\[
 \left[ \Delta_{J}(P) \Delta_{K}(Q) X^{J,K} \right]  = \left[\left(\Delta_{J}(P)X^J\right)  \otimes \left(\Delta_{K}(Q)X^K\right)  \right] = [x\otimes y]
.\]
 
For the second statement in the proposition, notice that for $[x\otimes y] \in Seg(Z_{p}\times Z_{q}) \cap U_{0}$, we have exhibited a matrix $A$ as in \eqref{block matrix} such that $\varphi^{n}([A,1]) = [x\otimes y]$.  But symmetric matrices are determined up to sign by their $1\times 1$ and $2 \times 2$ principal minors. Any other matrix must have the same blocked form as the one in \eqref{block matrix}.
\end{proof}

\begin{remark}
Proposition~\ref{prop:ZZ block} gives a useful tool in finding candidate modules for $I(Z_{n})$:  We are forced to consider 
\[I(Z_{n})\subset \bigcap_{\begin{array}{c}p+q=n \\ p,q\geq 1 \end{array}} I(Seg(Z_{p}\times Z_{q})).\]
\end{remark}

\section{Properties of the hyperdeterminantal module}\label{sec:module}
As a consequence of Theorem~\ref{thm:symmetry}, the defining ideal of $Z_{n}$, $\I(Z_{n}) \subset \Sym(V_{1}^{*}\otdot V_{n}^{*})$, is a $\G$-module. As mentioned above, we will consider the $\G$-module $HD =S_{(2,2)} S_{(2,2)} S_{(2,2)} S_{(4)} \dots  S_{(4)}$ (called the hyperdeterminantal module in \cite{HSt}).  In this section we compute the dimension of the hyperdeterminantal module and show that it occurs with multiplicity one in $S^{4}(V_{1}^{*}\otdot V_{n}^{*})$. Also, in the course of our observations, we arrive at a practical ideal membership test for a class of varieties that includes the variety of principal minors.

\begin{obs}
 The module $S_{(2,2)}\CC^{2}$ is 1-dimensional and the module $S_{(4)}\CC^{2}$ is $5$-dimen\-sional and therefore the dimension of the hyperdeterminantal module is
  \[\dim( S_{(2,2)} S_{(2,2)} S_{(2,2)} S_{(4)} \dots  S_{(4)})= \binom{n}{3}5^{n-3}.\]
\end{obs}

\begin{prop}\label{prop:multiplicity}
The module $HD =S_{(2,2)} S_{(2,2)} S_{(2,2)} S_{(4)} \dots  S_{(4)}$ occurs with multiplicity $1$ in $S^{4}(V_{1}^{*}\otdot V_{n}^{*})$. Moreover, $HD$ is an irreducible $\G$-module.
\end{prop} 
\begin{remark}
The fact that $HD$ occurs with multiplicity $1$ saves us a lot of work because we do not have to worry about which isomorphic copy of the module occurs in the ideal.
\end{remark}
\begin{proof}
%suggested change by referee
%For the ``moreover'' part, notice that the module $HD$ is the span of the $G$-orbit of a single polynomial (namely the hyperdeterminant of format $2\times 2 \times 2$ on the variables $X^{[i_{1},i_{2},i_{3},0, \dots ,0]}$ with $0\leq i_{1},i_{2},i_{3} \leq 1$) and therefore $HD$ is an irreducible module.
For the ``moreover'' part, notice that 
by definition, $HD= S_{(2,2)} S_{(2,2)} S_{(2,2)} S_{(4)} \dots  S_{(4)}$ is a direct sum over permutations yielding distinct $SL(2)^{\times n}$-modules.  It is a standard fact that each summand is an irreducible $SL(2)^{\times n}$-module, and this makes $HD$ an irreducible $\G$-module.
 
We need to examine the $SL(2)^{\times n}$-module decomposition of $S^4(V_{1}^{*} \otdot V_{n}^{*})$.  It suffices to prove for any fixed permutation $\sigma$, that $ S_{(2,2)}V_{\sigma(1)}^{*} \otimes S_{(2,2)}V_{\sigma(2)}^{*} \otimes S_{(2,2)}V_{\sigma(3)}^{*} \otimes S_{(4)} V_{\sigma(4)}^{*} \otimes  \dots  \otimes S_{(4)} V_{\sigma(n)}^{*}$ is an $SL(2)^{\times n}$-module which occurs with multiplicity $1$ in the decomposition of $S^4(V_{1}^{*} \otdot V_{n}^{*})$. 

We will follow the notation and calculations similar to  \cite{LM04}.
For a representation $[\pi]$ of the symmetric group $\mathfrak{S}_{d}$, let $\chi_{\pi}$ denote its character.  The number of occurrences of $S_{\pi_1}V_1^{*}\otdot S_{\pi_n}V_{n}^{*}$ in the decomposition of $S^d(V_1^{*} \otdot V_{n}^{*})$ is computed by the dimension of the space of $\mathfrak{S}_{d}$ invariants, $\dim\left(([\pi_1]\otdot[\pi_n])^{\mathfrak{S}_d}\right)$.  This may be computed by the formula  

\begin{equation*}
\dim\left(([\pi_1]\otdot[\pi_n])^{\mathfrak{S}_d}\right) = \frac{1}{d!}\sum_{\sigma \in\mathfrak{S}_d} \chi_{\pi_1}(\sigma)  \dots  \chi_{\pi_n}(\sigma)
.\end{equation*}

  In our case, we need to compute 
\begin{eqnarray*}
\dim\left(([(2,2)]\otimes[(2,2)]\otimes [(2,2)]\otimes[(4)]\otdot[(4)])^{\mathfrak{S}_4} \right)\\{}
= \frac{1}{4!}\sum_{\sigma \in\mathfrak{S}_4} \chi_{(2,2)}(\sigma) \chi_{(2,2)}(\sigma)\chi_{(2,2)}(\sigma)\chi_{(4)}(\sigma)  \dots  \chi_{(4)}(\sigma)
.\end{eqnarray*}
But, $\chi_{(4)}(\sigma)=1$ for every $\sigma\in\mathfrak{S}_4$.  So, our computation reduces to the following
\begin{eqnarray*}
 \dim\left(([(2,2)]\otimes[(2,2)]\otimes [(2,2)]\otimes[(4)]\otdot[(4)])^{\mathfrak{S}_{n}}\right) 
\\ = \frac{1}{4!}\sum_{\sigma \in\mathfrak{S}_4} \chi_{(2,2)}(\sigma) \chi_{(2,2)}(\sigma)\chi_{(2,2)}(\sigma) =1
,\end{eqnarray*}
where the last equality is found by direct computation.  The module $S_{(2,2)}V_1^{*} \otimes S_{(2,2)}V_2^{*} \otimes S_{(2,2)}V_3^{*}$ occurs with multiplicity $1$ in $S^4(V_1^{*}\otimes V_2^{*} \otimes V_3^{*})$. (The full decomposition of $S^4(V_1^{*}\otimes V_2^{*} \otimes V_3^{*})$ was computed in (prop 4.3 \cite{LM04}).)  Therefore the module $ S_{(2,2)}V_{\sigma(1)}^{*} \otimes S_{(2,2)}V_{\sigma(2)}^{*} \otimes S_{(2,2)}V_{\sigma(3)}^{*} \otimes S_{(4)} V_{\sigma(4)}^{*} \otimes  \dots  \otimes S_{(4)} V_{\sigma(n)}^{*}$ occurs with multiplicity $1$ in $S^{4}(V_{1}^{*}\otdot V_{n}^{*})$. 

We have seen that each summand of $HD$ is an irreducible $SL(2)^{\times n}$-module which occurs with multiplicity $1$ in $S^4(V_{1}^{*}\otdot V_{n}^{*})$. Therefore $HD$ is an irreducible $\G$-module, and it occurs with multiplicity $1$ in $S^4(V_{1}^{*}\otdot V_{n}^{*})$.
\end{proof}
We remark that the above argument generalizes to:
\begin{lemma}\label{lem:multiplicity:general}
 For every collection $\pi_{1}, \dots ,\pi_{n}$ of partitions of $d$, 
\begin{equation*}
 \dim\left(([\pi_1]\otdot[\pi_n])^{\mathfrak{S}_d}\right) =
 \dim\left(([\pi_1]\otdot[\pi_n]\otimes[(d)])^{\mathfrak{S}_{d}}\right)
.\end{equation*}
In particular, if $M$ is any irreducible $SL(V_{1})\tdt SL(V_{n})$-module which occurs with multiplicity $m$ in $S^{d}(V_{1}^{*}\otdot V_{n}^{*})$, then $M\otimes S^{d}V_{n+1}^{*}$ is an irreducible $SL(V_{1})\tdt SL(V_{n})\times SL(V_{n+1})$-module which occurs with multiplicity $m$ in $S^{d}(V_{1}^{*}\otdot V_{n}^{*}\otimes V_{n+1}^{*})$. 
\end{lemma}
\begin{proof}
Use the formula
\begin{equation*}
\dim\left(([\pi_1]\otdot[\pi_n])^{\mathfrak{S}_d}\right) = \frac{1}{d!}\sum_{\sigma \in\mathfrak{S}_d} \chi_{\pi_1}(\sigma)  \dots  \chi_{\pi_n}(\sigma)
.\end{equation*} 
and note that $\chi_{(d)}(\sigma)=1$ for every $\sigma \in \mathfrak{S}_d$.
\end{proof}

\begin{prop}\label{prop:one way} The hyperdeterminantal module is contained in the ideal of the variety of principal minors of symmetric matrices, \ie
\[HD= S_{(2,2)} S_{(2,2)} S_{(2,2)} S_{(4)} \dots  S_{(4)} \subseteq  \I(Z_{n}) ,\]
and in particular, $Z_{n}\subseteq \V(HD)$.
\end{prop}
\begin{proof}
Note, this statement is proved in \cite{HSt}.  The following is a slightly different proof that uses representation theory.
Both $HD$ and $\I(Z_{n})$ are $\G$-modules and $HD$ is an irreducible $\G$-module, so we only need to show that the highest weight vector of $HD$ vanishes on all points of $Z_{n}$.  The highest weight vector of $HD$ is the hyperdeterminant of format $2 \times 2 \times 2$ on the variables $X^{[i_{1},i_{2},i_{3},0, \dots ,0]}$.  The set \[Z_{n}\cap span\{X^{[i_{1},i_{2},i_{3},0, \dots ,0]} \mid i_{1},i_{2},i_{3}\in \{0,1\}\},\] is the set of principal minors of the upper-left $3\times 3$ block of $n\times n$ matrices. The highest weight vector of $HD$ vanishes on these principal minors because of the case $n=3$, so there is nothing more to show.
\end{proof}

\begin{prop}\label{prop:unaugment}  Let $V$ and $W$ be complex vector spaces with $\dim(V)\geq 2$. Suppose $Y\subset \PP W$ and $X\subset \PP ( V \otimes W )$ are varieties such that $Seg(Y\times \PP W) \subset X$. 
 Suppose $M\subset S^{d}V^{*}$ is a space of polynomials. Then $M\otimes S^{d}W^{*} \subset \I_{d}(X)$ only if $M \subset \I_{d}(Y)$. 
\end{prop}
\begin{proof}
By Lemma~\ref{lem:multiplicity:general}, it makes sense to think of $M\otimes S^{d}W^{*}\subset S^{d}(V^{*}\otimes W^{*})$ as a space of polynomials.

There exists a basis of $S^{d}W^{*}$ of vectors of the form $\alpha^{d}$. So $M\otimes S^{d}W^{*}$ has a basis of vectors of the form $f\otimes \alpha^{d}$ with $f\in M$ and $\alpha \in W^{*}$.  It suffices to prove the proposition on this basis.

Suppose $f\otimes \alpha^{d}$ is a basis vector in $M\otimes S^{d}(W^{*}) \subset \I_{d}(X)$.  Then $Seg(Y\times \PP W) \subset X$ implies that $f\otimes \alpha^{d}\in \I_{d}(Seg(Y\times \PP W)) \subset S^{d}(V^{*}\otimes W^{*})$.  This means that $f\otimes \alpha^{d}(y\otimes w) = 0$ for all $y\in Y$ and for all $w\in W$.  It is a fact that $\alpha^{d}(w) = \alpha(w)^{d}$ (this can be deduced from Lemma~\ref{lem:polar} below, for instance), so we can evaluate
\[f\otimes \alpha^{d}(y\otimes w) = f(y) \alpha^{d}(w) = f(y) \alpha(w)^{d}.\]  Since $\dim(V) \geq 2$, $\V(\alpha)$ is a hyperplane. It is no problem to choose a point that   misses a hyperplane, so we can choose a particular $w\in W$ so that $\alpha(w) \neq 0$.

So we have $f(y) \alpha(w)^{d} = 0$ for all $y\in Y$ and $\alpha(w)\neq 0$, so $f(y)=0$ for all $y\in Y$ and hence $f \in \I_{d}(Y)$.  We can repeat the argument for any $f\in M$ we choose, so we are done.
\end{proof}

Proposition~\ref{prop:unaugment} fails to be an if and only if statement.  Explicitly, we \emph{cannot} say that every module in the space $\I_{d}(X)$ occurs as $M\otimes S^{d}V^{*}$ for a subset $M \subset I_{d}(Y)$.  In Section~\ref{sec:augment} we study the zero sets of modules of the form $I_{d}(Y)\otimes S^{d}V^{*}$, and this sheds light on the failure of the converse of Proposition~\ref{prop:unaugment}

\begin{remark}
Proposition~\ref{prop:subvar} says that $Seg(Z_{n}\times \PP V_{n+1}) \subset Z_{n+1}$. We can use this proposition to study the variety of principal minors in two ways.
First, if $M$ is a module in $\I_{d}(Z_{n})$, then $M\otimes S^{d}V_{n+1}$ is a module in $\I_{d}(Z_{n+1})$. The second use is the contrapositive version. It gives an easy test for ideal membership for modules that have at least one $S_{(d)}V_{i}^{*}$ factor.  Suppose we know $\I_{d}(Z_{n})$ for some $n$.  If we want to test whether $M=S_{\pi_{1}}V_{1}^{*} \otdot S_{\pi_{n+1}} V_{n+1}^{*}$ is in $\I_{d}(Z_{n+1})$ and we know that $M$ has at least one $\pi_{i}=(d)$, then we can remove $S_{\pi_{i}}V_{i}^{*}$ and check whether the module we have left is in $\I_{d}(Z_{n})$.
\end{remark}

\section{A geometric characterization of the zero set of the hyperdeterminantal module via augmentation}\label{sec:augment}

The hyperdeterminantal module has a useful inductive description that we would like to be able to exploit.  In particular,  for $n\geq 3$, the module is always of the form 
\[S_{(2,2)}S_{(2,2)}S_{(2,2)}S_{(4)}\dots S_{(4)},\]
where the number of $S_{(4)}$ factors is $n-3$. So for $n\geq 3$, to construct the $\G$-module $HD$ in the case $n=k+1$ from $HD$ in the case $n = k$, we simply append another $S^{4}$.  

More generally, if $M$ is an $SL(V)$-module, we will call a $SL(V)\times SL(W)$-module of the form $M\otimes S^{d}W^{*}$ an \emph{augmentation} or \emph{augmented module}.  So for $n\geq 4$, $HD$ can be considered as the sum of augmented modules.

In this section, we study augmented modules and their zero-sets in order to arrive at a geometric description of the zero set of an augmented module (Lemma~\ref{lem:augment}).  By using applying this geometric description to the hyperdeterminantal module, we get a geometric description of its zero set (Lemma~\ref{lem:characterization}).  This description is essential in our proof of Theorem~\ref{thm:Main Theorem}.

\subsection{Polarization and its application to augmented modules}

Augmentation is similar to prolongation, a concept found in the study of the ideals of secant varieties.  A difference between the two is that augmentation does not change the degree of the polynomials, whereas prolongation increases the degree.

It is not a surprise that we can get inspiration from the techniques used to study secant varieties when studying augmented modules.
In particular, polarization is a tool from classical invariant theory \cite[p.5,6]{Weyl} and is useful in the study ideals of secant varieties (see \cite{LM03,SS} for recent examples). In what follows, we use polarization to better understand the polynomials in an augmented module.

Polarization allows for the passage from a homogeneous polynomial to a symmetric multi-linear form.
Let $x_{1},\dots,x_{n}$ be a basis of $V$, and let $v_{i}=v_{i,1}x_{1}+\dots+v_{i,n}x_{n}$ for $1\leq i\leq d$. Given a homogeneous degree $d$ polynomial $f$ in the polynomial ring $\CC[x_{1},\dots, x_{n}]$,
the \emph{polarization} of $f$ is a symmetric multi-linear form $\overrightarrow{f} \in S^{d}V^{*}$ where we define $\overrightarrow{f}(v_{1},\dots,v_{d})$ to be the coefficient of $t_{1}t_{2}\dots t_{d}$ in the expansion of
\[
f(t_{1}v_{1}+\dots+t_{d}v_{d})
\]
considered as a polynomial in $t_{1},\dots,t_{d}$.  For example if $f(x_{1},x_{2}) = (x_{1})^{2}x_{2}$ one calculates that $\overrightarrow{f}(v_{1},v_{2},v_{3}) = 2(v_{1,1}v_{2,1}v_{3,2} + v_{1,1} v_{3,1} v_{2,2} + v_{2,1} v_{3,1} v_{1,2})$.

The following useful characterization is straightforward from the definition, and while it is a consequence of classical facts \cite{Weyl}, we found it stated in \cite{SS}.  

\begin{lemma}[Lemma 2.5(1) \cite{SS}]\label{lem:SSLemma}
If $F$ is a homogeneous degree $d$ polynomial in $x_{1},\dots,x_{n}$, let $\overrightarrow{F}$ denote its polarization.   Let $v = t_{1}x_{1}+ \dots + t_{k}x_{k}$.  Then 
\begin{equation}\label{evaluate}F(v) = \overrightarrow{F}(v, \dots ,v)
=\sum_{\mathbf{\beta}} \frac{1}{\mathbf{\beta} !}\mathbf{t^{\beta}} \overrightarrow{F}\left(\mathbf{x^{\beta}}\right)
,\end{equation}
where $\mathbf{\beta}=(\beta_{1}, \dots ,\beta_{k})$, is a (non-negative) partition of $d$, $\mathbf{\beta}!=\beta_{1}! \dots \beta_{k}!$, $\mathbf{t^{\beta}}=  t_{1}^{\beta_{1}} \dots  t_{k}^{\beta_{k}}$, 
and $ \overrightarrow{F}\left(\mathbf{x^{\beta}}\right)= \overrightarrow{F}\left(x_{1}^{\beta_{1}}, \dots ,x_{k}^{\beta_{k}}\right)$, and $x_{i}^{\beta_{i}}$ is to be interpreted as $x_{i}$ repeated $i$ times.
\end{lemma}
Here is an example of the utility of this lemma that we will need later.
\begin{lemma}\label{lem:linear condition}
A linear space $L = span\{x_{1},\dots,x_{k}\}$ is a subset of $\V(f)$ if and only if $\overrightarrow{f}\left(\mathbf{x^{\beta}}\right) = 0$ for every partition $\beta$ of $d$.
\end{lemma}
\begin{proof}
A linear space $L = span\{x_{1},\dots,x_{k}\}$ is in the zero set of $f$ if and only if $f(t_{1}x_{1}+\dots+ t_{k}x_{k}) = 0$ for all choices of $t_{i}\in \CC$. Formula \eqref{evaluate} says that
\begin{equation}\label{evaluate2}f(t_{1}x_{1}+\dots+ t_{k}x_{k})  =\sum_{\mathbf{\beta}} \frac{1}{\mathbf{\beta} !}\mathbf{t^{\beta}} \overrightarrow{f}\left(\mathbf{x^{\beta}}\right)
,\end{equation}
and thus implies that if $\overrightarrow{f}\left(\mathbf{x^{\beta}}\right) = 0$ for all $\beta$ then $f(t_{1}x_{1}+\dots+ t_{k}x_{k}) =0$ for all $t_{i}\in \CC$.

For the other direction, suppose  $f(t_{1}x_{1}+\dots+ t_{k}x_{k}) =0$ for all $t_{i}\in \CC$. Consider a fixed partition $\mathbf{\beta'}$ and take the derivative $\frac{\partial}{\partial\mathbf{t^{\beta'}}}$ of \eqref{evaluate2} to get
\[
0= \overrightarrow{f}\left(\mathbf{x^{\beta'}}\right) +\sum_{\mathbf{\beta}>\mathbf{\beta'}} \frac{1}{(\mathbf{\beta-\beta'}) !}\mathbf{t^{\beta-\beta'}} \overrightarrow{f}\left(\mathbf{x^{\beta}}\right)
.\]
Then take limits, as $t_{i}\rightarrow 0$ to find that $0=\overrightarrow{f}\left(\mathbf{x^{\beta'}}\right)$.  We do this for each $\mathbf{\beta'}$ to conclude. 
\end{proof}

In general, the polarization of the tensor product of two polynomials is not likely to be the product of the polarized polynomials; however, there is something we can say in the following special case:
\begin{lemma}\label{lem:polar}
 Let $F\in S^d(W^{*})$ and let $\overrightarrow{F}$ denote its polarization. Then for $\gamma\in V^{*}$ we have
\[\overrightarrow{F\otimes (\gamma)^d} = \overrightarrow{F}\otimes \overrightarrow{(\gamma)^d} = \overrightarrow{F}\otimes (\gamma)^d.\] 
\end{lemma}
\begin{proof}
A standard fact about the polarization is that $\overrightarrow{F}$ is a symmetric multi-linear form.  It is obvious that $\overrightarrow{(\gamma)^d}= (\gamma)^d$, because $(\gamma)^d$ is already symmetric and multi-linear.

So it remains to prove the first equality in the lemma, which we will do by induction on the number of terms in $F$.  Suppose $F$ is a monomial, $F = \mathbf{w}^{\mathbf{\alpha}} = w_{1}^{\alpha_{1}}\circ \dots \circ w_{n}^{\alpha_{n}}$. Then use the isomorphism $W^{\otimes d}\otimes V^{\otimes d}\simeq (W\otimes V)^{\otimes d}$, and write $\mathbf{w}^\alpha \otimes \gamma^d =  (w_{1}^{\alpha_{1}}\otimes \gamma^{\alpha_{1}})\circ \dots \circ (w_{n}^{\alpha_{n}}\otimes \gamma^{\alpha_{n}})  = (w_{1}\otimes \gamma)^{\alpha_{1}}\circ \dots \circ (w_{n}\otimes \gamma)^{\alpha_{n}} = (\mathbf{w}\otimes \gamma)^{\mathbf{\alpha}}$.

 If $F$ is not a monomial, suppose $F= F_{1}+ F_{2}$ with $F_{i}$ nonzero polynomials for $i=1,2$ each having strictly fewer monomials than $F$. It is clear that $\overrightarrow{F_{1}+F_{2}} = \overrightarrow{F_{1}}+ \overrightarrow{F_{2}}$. Also, the operation $\otimes \gamma^{d}$ is distributive.  So $\overrightarrow{F\otimes (\gamma)^d} = \overrightarrow{F_{1}\otimes \gamma^{d}}+ \overrightarrow{F_{2}\otimes \gamma^{d}}$.  By the induction hypothesis, we know that  $\overrightarrow{F_{i}\otimes \gamma^{d}} =  \overrightarrow{F_{i}}\otimes \gamma^{d}$ for $i=1,2$. We conclude that $\overrightarrow{F_{1}\otimes \gamma^{d}}+ \overrightarrow{F_{2}\otimes \gamma^{d}} = (\overrightarrow{F_{1}}+\overrightarrow{F_{2}})\otimes \gamma^{d} = \overrightarrow{F}\otimes \gamma^{d}$.
\end{proof}

The following lemma was inspired by methods found in \cite{LM03}. It is a geometric description of the zero set of an augmented module.
\begin{lemma}[Augmentation Lemma]\label{lem:augment}
Let $W$ and $V$ be complex vector spaces with $\dim(V) \geq 2$. Let $X\subset \PP W$ be a variety and let $\I_{d}(X) = \I(X)\cap S^{d}W^{*}$ be the vector space of degree $d$ polynomials in the ideal $\I(X)$. Then 
\begin{equation}\label{step}
\V(\I_{d}(X)\otimes S^{d}V^{*})=Seg(\V(\I_{d}(X)) \times \PP V) \cup \bigcup_{L\subset \V(\I_{d}(X))}\PP (L \otimes  V), 
\end{equation}
where $L\subset  \V(\I_{d}(X))$ are linear subspaces.
\end{lemma}
Note that since the linear spaces $L$ can be one dimensional, we do have 
\[
Seg(\V(\I_{d}(X)) \times \PP V) \subset \bigcup_{L\subset \V(\I_{d}(X))}\PP (L \otimes V) 
,\]
 and we will use Lemma~\ref{lem:augment} with the two terms on the right hand side of \eqref{step} combined, but we keep the two parts separate for emphasis here.

\begin{remark} Note that if $\I(X)$ is generated in degree no larger than $d$, then one can replace $ \V(\I_{d}(X))$ with $X$ in the statement of Lemma~\ref{lem:augment}.  We will use the result of Lemma~\ref{lem:augment} with the induction hypothesis  that $\V(HD) = Z_{n}$ and obtain a description of the zero set $\V(HD\otimes S^{4}V_{n+1})$ in terms of the geometry of $Z_{n}$.
\end{remark}

\begin{proof}[Proof of Lemma~\ref{lem:augment}]
First we prove ``$\supseteq$''.
Suppose $\dim(V)=n\geq 2$. Recall that we can choose a basis of  $S^{d}V^{*}$ consisting of $d^{th}$ powers of linear forms, $\{(\gamma_{1})^{d}, \dots ,(\gamma_{r})^{d}\}$, where $r = \binom{n+d-1}{d}$ and the $\gamma_{i}$ are in general linear position.  It suffices to work on a basis of the vector space $\I_{d}(X)\otimes S^{d}V^{*}$.  We choose a basis consisting of polynomials of the form $f\otimes \gamma^{d}$, with $f\in \I_{d}(X)$ and $\gamma \in V^{*}$.

Suppose $[x\otimes a]\in Seg( \V(\I_{d}(X)) \times \PP V)$ and evaluate $(f\otimes \gamma^{d})(x\otimes a)  = f(x)\gamma^{d}(a)$.  But $x \in  \V(\I_{d}(X))$, so $f(x)=0$ for every $f\in \I_{d}(X)$, and in particular, $[x\otimes a] \in \V(\I_{d}(X)\otimes S^{d}V^{*})$. So we have established that $\V(\I_{d}(X)\otimes S^{d}V^{*})\supset Seg( \V(\I_{d}(X)) \times \PP V)$.

Now suppose $[v]\in \PP (L \otimes V)$ for some linear subspace $L=span\{x_{1}, \dots ,x_{l}\} \subset  \V(\I_{d}(X))$.  By expanding an expression of $[v]$ in bases and collecting the coefficients of the $x_{i}$, we can write $[v]=[x_{1}\otimes a_{1}+ \dots +x_{l}\otimes a_{l}]$ for $ a_{i}\in  V$ not all zero.  Consider $f\otimes \gamma^{d} \in \I_{d}(X)\otimes S^{d}V$.
By Lemma~\ref{lem:polar}, $\overrightarrow{f\otimes \gamma^{d}}=\overrightarrow{f}\otimes \gamma^{d}$ and using the polarization formula \eqref{evaluate}, we write
\[\left(f\otimes \gamma^{d}\right)(v) = \left(\overrightarrow{f}\otimes \gamma^{d}\right)(v, \dots ,v) = \sum_{\mathbf{\beta}} \frac{1}{\mathbf{\beta}!}\overrightarrow{f}(\mathbf{x^{\beta}}) \gamma^{d}\mathbf{(a^{\beta})}
.\]
The choice of $L\subset \V(\I_{d}(X))$ means that $L \subset \V(f)$, so by Lemma~\ref{lem:linear condition}, $\overrightarrow{f}(\mathbf{x^{\beta}}) = 0$ for all $\mathbf{\beta}$. Every term of $(f\otimes \gamma^{d})(v)$ vanishes so $(f\otimes \gamma^{d})(v)=0$, and hence $[v]\in \V(\I_{d}(X)\otimes S^{d}V^{*})$.
So we have established that $\V(\I_{d}(X)\otimes S^{d}V^{*})\supset \PP (L \otimes V) 
$ for all linear subspaces $L\subset \V(\I_{d}(X))$.

Now we prove ``$\subseteq$''.
Consider any $[v]\in \PP(W\otimes V)$. Choose a basis $\{a_{1},\dots, a_{k}\}$ of $V$ (by assumption $k\geq 2$). Then expand the expression of $v$ in bases and collect the coefficients of each $a_{i}$ to find $[v] = [x_{1}\otimes a_{1}+\dots + x_{k}\otimes a_{k}]$ with $x_{1},\dots, x_{k} \in W$ and not all $x_{i}$ zero.  
%This is the statement that $\sigma_{k}(\PP W \times \PP V) = \PP (W\otimes V)$ when $k = \dim(V)$.

We need to show that $[v] \in \PP (L \otimes  V)$ for a linear space $L \subset \V(\I_{d}(X))$.  The natural linear space to consider is $L=span\{x_{1},\dots, x_{k}\}$.  Since we already have an expression $[v]=[x_{1}\otimes a_{1}+ \dots + x_{k}\otimes a_{k}]$, if we can show that $L=span\{x_{1},\dots, x_{k}\} \subset  \V(\I_{d}(X))$, we will be done.

For any $f\otimes \gamma ^{d} \in \I_{d}(X)\otimes S^d V^*$  we can write
\begin{equation}\label{star}
0 =(f\otimes \gamma ^{d})(v)= \sum_{\mathbf{\beta}} \frac{1}{\mathbf{\beta}!} \overrightarrow{f}(\mathbf{x^{\beta}}) \gamma^{d}\mathbf{(a^{\beta})}
.\end{equation}
Let $\{\check{a_{1}},\dots,\check{a_{k}}\}$ be basis of $V^{*}$ dual to $\{a_{1},\dots, a_{k}\}$.  Then let $\gamma$ vary continuously in $V^{*}$ by writing it as
\[
\gamma = t_{1}\check{a_{1}}+\dots+t_{k}\check{a_{k}}
\]
where the parameters $t_{i}\in \CC$ vary.
The polynomial $\gamma^{d}$ is simple enough that we can expand it as follows:
\[
\gamma^{d}\mathbf{(a^{\beta})} = \gamma^{d}(a_{1}^{\beta_{1}},\dots,a_{k}^{\beta_{k}}) = \gamma(a_{1})^{\beta_{1}}\dots \gamma(a_{k})^{\beta_{k}}
\]
But our choices have made it so that $\gamma(a_{i}) = t_{i}$, and therefore $\gamma^{d}(\mathbf{a^{\beta}}) = \mathbf{t^{\beta}}$.  So \eqref{star} becomes
\[
0 =(f\otimes \gamma ^{d})(v)= \sum_{\mathbf{\beta}} \frac{1}{\mathbf{\beta}!} \overrightarrow{f}(\mathbf{x^{\beta}}) \mathbf{t^{\beta}} = f(t_{1}x_{1}+\dots+t_{k}x_{k})
,\]
where we have used Lemma~\ref{lem:polar}.  So $f(t_{1}x_{1}+\dots+t_{k}x_{k})=0$ for all $t_{i}\in \CC$ and this is an equivalent condition that $L=span\{x_{1},\dots,x_{k}\}$ is a subspace of $\V(f)$.  Since this was done for arbitrary $f\in \I_{d}(X)$, we conclude that $L\subset \V(\I_{d}(X))$.
\end{proof}

Now we can apply this geometric characterization of augmentation to the hyperdeterminantal module.  To do this we need to set up more notation.

Assume $n\geq 4$. Let $HD_{i}$ be the image of the hyperdeterminantal module at stage $n-1$ under the following re-indexing isomorphism \[
S^{4}(V_{1}^{*}\otdot V_{n-1}^{*}) \longrightarrow S^{4}(V_{1}^{*}\otdot V_{i-1}^{*}\otimes V_{i+1}^{*}\otdot V_{n}^{*}),
\]
where we still have $n-1$ vector spaces $V_{i}\simeq \CC^{2}$, but we have shifted the index on the last $n-i$ terms.
Then the hyperdeterminantal module at stage $n$ can be expressed as a sum of augmented modules as follows: 
\[HD=\sum_{i=1}^{n} (HD_{i}\otimes S^{4}V_{i}^{*}).\]

Finally note that if $\dim(V)=k$, then $\sigma_s(\PP W\times \PP V) = \PP(W\otimes V)$ for all $s\geq k$.  In the case $V_{i}\simeq \CC^{2}$, we have $\PP( L \otimes V_{i}) = \sigma_{2}(\PP L \times \PP V_{i})$.  Certainly 
\[
Seg(\V(M_{i}) \times \PP V_{i}) \subset \bigcup_{L\subset  V(M_{i})} \PP( L \otimes V_{i})  
,\] for any modules of polynomials $M_{i}$.
If $L\subset  \V(\I_{d}(X))$, then $\sigma_s(\PP L\times \PP V) \subseteq \sigma_s( \V(\I_{d}(X))\times \PP V)$. 
If $A,B,C$ are vector spaces of polynomials such that $C=A + B$ then $\V(C) = \V(A)\cap \V(B)$. Collecting these ideas, we apply the Augmentation Lemma~\ref{lem:augment} to the hyperdeterminantal module to yield the following:

\begin{lemma}[Characterization Lemma]\label{lem:characterization}
Consider $\sum_{i=1}^{n} HD_{i}\otimes S^{d} V_{i}^{*} \subset S^{d}(V_{1}^{*}\otdot V_{n}^{*})$. Then 
\begin{equation*} 
\V\left(\sum_{i=1}^{n} HD_{i}\otimes S^{d} V_{i}^{*} \right) 
=\bigcap_{i=1}^{n} \left(\bigcup_{L\subset  V(HD_{i})} \PP( L \otimes V_{i})  \right)
\subseteq  
\bigcap_{i=1}^{n} \left( \sigma_{2} (\V(HD_{i}) \times \PP V_{i}) \right).
\end{equation*}
\end{lemma}

\begin{remark} A consequence of the characterization lemma is the following test for non-membership in the zero-set of $HD$. Suppose $[z]=[\zeta^{1}\otimes x_{i}^{1} + \zeta^{2}\otimes x_{i}^{2}]\in \PP^{2^n-1}$.  If either $[\zeta^{1}] $ or  $[\zeta^{2}]$ is not a vector of principal minors of an $(n-1) \times (n-1)$ symmetric matrix, then $[z]$ is not a zero of the hyperdeterminantal module $HD$ and hence not a vector of principal minors of a symmetric matrix since $\V(HD)\supset Z_{n}$. This observation can be iterated, and each iteration cuts the size of the vector in question in half until one only need to check honest hyperdeterminants of format $2 \times 2 \times 2$.  This test, while relatively cheap and accessible, is necessary but not sufficient as is pointed out in \cite{HSt}.

It is well known that the ideal of the Segre product of an arbitrary number of projective spaces is generated in degree $2$ by the $2 \times 2$ minors of flattenings.  In essence, this is saying that all of the polynomials in the ideal come from the Segre products of just two projective spaces.  The following is a weaker, strictly set-theoretic result in the same spirit. It is another application of the Augmentation Lemma~\ref{lem:augment}, and its proof is mimicked in the proof of Lemma~\ref{lem:one piece n} below.
\begin{prop}\label{prop:segre ideal} For $1\leq i \leq n$, let $V_{i}$ be complex vector spaces each with dimension $\geq 2$ and assume $n\geq 2$.
If for each $i$, $B_{i}\subset S^d(V_{1}^*\otimes  \dots V_{i-1}^*\otimes V_{i+1}^* \otimes   \dots  \otimes V_{n}^{*})$ is a set of polynomials with the property 
\[
\V(M^{i})= Seg (\PP V_{1}\times \dots   \PP V_{i-1}\times  \PP V_{i+1}\times \dots  \times \PP V_{n}),\] then
\[
\V\left(\bigoplus_{i} (M_{i}\otimes S^d V_{i}^{*})\right) = Seg (\PP V_{1}\times \dots  \times \PP V_{n})
.\]
\end{prop}
\begin{proof}
Work by induction and use the Augmentation Lemma~\ref{lem:augment}. It is clear that $\V(\bigoplus_{i} (M_{i}\otimes S^d V_{i}^{*})) \supset Seg (\PP V_{1}\times \dots  \times \PP V_{n})$.  All the linear spaces on $Seg (\PP V_{1}\times \dots  \times \PP V_{n})$ are (up to permutation) of the form $V_{1}\otimes \widehat{a_{2}}\otdot \widehat{a_{n}}$ where $a_{i}\in V_{i}$ are nonzero and $\widehat{a_{i}}$ denotes the line through $a_{i}$.  Then compute the intersection, $\bigcup_{L^i} \bigcap_{i=1}^n \PP(L^i\otimes V_{i})$, and notice that in the intersection of just 3 factors, all of the resulting linear spaces must live in $Seg (\PP V_{1}\times \dots  \times \PP V_{n})$.
\end{proof}

\end{remark}
%%%%%%%%%%%%%%  ^^^^^^^^^  Geometry ^^^^^^^^^     %%%%%%%%%%%%%%%%%%%

%%%%%%%%%%%%%%  Proof  %%%%%%%%%%%%%%%%%%%

\section{Understanding the case when two zeros of the hyperdeterminantal module disagree in precisely one coordinate.}\label{sec:understanding}

In the proof of Theorem~\ref{thm:Main Theorem} below we work to construct a matrix whose principal minors are a given point in the zero set of the hyperdeterminantal module.  The main difficulty is the following.
Suppose we have a point $[z]\in \V(HD)$ and a candidate matrix $A$ that satisfies $\Delta_{I}(A) = z_{I}$ for all $I\neq [1, \dots ,1]$.  In other words, all of the principal minors of $A$ except possibly for the determinant agree with the entries of $z$.  What can we say about $z$?

To answer this question, we must study the points in $\widehat{\V(HD)}$ that have all of their coordinates except one equal.  Geometrically, we need to understand the points for which a line in the coordinate direction $X^{[1,\dots,1]}$ above the point $z$ intersects $\widehat{\V(HD)}$ in at least two points.  We answer this question in Lemma~\ref{lem:almost} below.  Using that lemma, we find the following
\begin{prop}\label{prop:no almost}
 Let $n\geq 4$. Suppose $z = z_{I}X^{I}$ and $w = w_{I}X^{I}$ are points in $\widehat{\V(HD)}$.  If $z_{I}=w_{I}$ for all $I\neq [1, \dots ,1]$ and $z_{[0,\dots,0]} \neq 0$, then $z = w$.
\end{prop}

For the rest of this section will use the following notation.  If $K=\{k_{1},\dots,k_{s}\} \subset \{1,\dots,n\}$ and $1\leq k_{j}\leq n$ for all $j$, then let $V_{K}\simeq V_{k_{1}}\otdot V_{k_{s}}$ for $s\leq n$.  We assume $V_{k}\simeq \CC^{2}$ for all $k$, so that $V_{K} \simeq (\CC^{2})^{\otimes s}$.  Let $\mathcal{P}^{2}(\{n_{1},\dots,n_{n}\})$ denote the collection of all partitions of $\{n_{1},\dots,n_{n}\}$ into mutually disjoint subsets of cardinality $2$ or less, \ie $\mathcal{P}^{2}(\{1,\dots,n\})$ consists of the sets $\{K_{1},\dots,K_{m}\}$ such that $K_{p}\subset \{1,\dots,n\}$ and $|K_p| \leq 2$ for every $1\leq p \leq m$,  $K_{p}\cap K_{q}=\emptyset$ whenever $p\neq q$, and $\cup_{p=1}^{m} K_{p}=\{1,\dots,n\}$.

%%%%%%%%%%%%%%%%% Almost %%%%%%%%%%%%%%%%%%%%%%
\begin{lemma}\label{lem:almost}  Let $n\geq 4$. Suppose $z = z_{I}X^{I}$ and $w = w_{I}X^{I}$ are points in $\widehat{\V(HD)}$ . If $z_{I}=w_{I}$ for all $I\neq [1, \dots ,1]$ but $z_{[1, \dots ,1]}\neq w_{[1, \dots ,1]}$, then 
\[
[z],[w]\in \bigcup_{\{K_{1},\dots,K_{m}\}\in \mathcal{P}^{2}(\{1,\dots,n\})} 
 Seg\left( \PP V_{K_{1}} \tdt \PP V_{K_{m}} \right)  
\subset Z_{n}.\]
\end{lemma}
Note that the notationally dense Segre product is just a product of $\PP^{3}$'s and $\PP^{1}$'s.

\begin{proof}[Proof of Proposition~\ref{prop:no almost}] Assume Lemma~\ref{lem:almost}. Let $z = z_{I}X^{I}$ and $w = w_{I}X^{I}$ be points in $\widehat{\V(HD)} \cap \{z\mid z_{[0,\dots,0]} \neq 0\}$. Suppose that $z_{I}=w_{I}$ for all $I\neq [1, \dots ,1]$, and suppose for contradiction that $z_{[1, \dots ,1]}\neq w_{[1, \dots ,1]}$. Lemma~\ref{lem:almost} implies that $[z],[w]$ are in a Segre product of $\PP^{1}$'s and $\PP^{3}$'s.

Note that $Z_{1} \simeq \PP^{1}$ and $Z_{2}\simeq \PP^{3}$ and Proposition~\ref{prop:ZZ block} implies that a point $[A,t]$ with $t\neq 0$ mapping to $Seg\left( \PP V_{K_{1}} \tdt \PP V_{K_{m}} \right)$ with $\{K_{1},\dots,K_{m}\} \in \mathcal{P}^{2}(\{1,\dots,n\})$  is permutation equivalent to a block diagonal matrix consisting of $1\times 1$ and $2 \times 2$ blocks.  Moreover, such a block diagonal matrix is a special case of a symmetric tri-diagonal matrix, and therefore none of its principal minors depends on the sign of the off-diagonal terms.  So fixing the $0\times 0$, $1\times 1$ and $2 \times 2$ principal minors fixes the rest of the principal minors in such a matrix. If we take $z_{[0,\dots,0]} = w_{[0,\dots,0]} = 1$ and assume the $1\times 1$ and $2 \times 2$ principal minors agree, then the rest of the principal minors must agree, including the  determinants, thus the contradiction.

Note that the assumption $z_{[0,\dots,0]} \neq 0$ is necessary.  If $z_{[0,\dots,0]}=0$, then consider the image of any two matrices $A,B$ with different nonzero determinants under the principal minor map with $t=0$. Then $\varphi([A,0]) = [0,\dots, 0,\det(A)] \neq \varphi([B,0]) = [0,\dots,0,\det(B)]$.
\end{proof}

\begin{remark}
A key point here is that we are not making the claim in Proposition~\ref{prop:no almost} for $n=3$.  In this case any two zeros of the hyperdeterminant are principal minors of  $3\times 3$ matrices which differ up to sign of the off-diagonal terms. Altering the sign of the off-diagonal terms of a $3\times 3$ symmetric matrix can change the determinant without changing the other principal minors and without forcing the matrix to be blocked as a $2 \times 2$ block and a $1\times 1$ block.
\end{remark}
\begin{remark}
To see that the analog of Proposition~\ref{prop:no almost} holds for $Z_{n}$ with $n\geq 4$ and $t\neq 0$ requires much less work than the case of $\V(HD)$.  We used Maple to construct a generic symmetric $4\times 4$ matrix and computed its principal minors. Then we changed the signs of the off-diagonal terms in every possible combination and compared the number of principal minors that agreed with the principal minors of the original matrix.  The result was that the two vectors of principal minors could agree in precisely $11, 13$ or $16$ entries, but not $15$. (Though tedious, the $4\times 4$ case can also be proved without a computer by analyzing the parity of the various products of the off-diagonal terms in the matrix.) We repeated the experiment in the $5\times 5$ case and found that the two vectors could agree in precisely $16, 19, 20, 21, 23, 25$ or $32$ positions, but never $31$ positions.

The general case follows from the $4\times 4$ case by the following.  Suppose $n\geq 4$ and $2^{n}-1$ of the principal minors of an $n\times n$ symmetric matrix agree with the principal minors of another $n\times n$ symmetric matrix. Then we may assume that the $0\times 0$, $1\times 1$ and $2 \times 2$ principal minors of both matrices agree and hence the matrices must agree up to the signs of the off-diagonal terms.  Then use the group to move the one position where the principal minors don't agree to be a $4\times 4$ determinant and use the $4\times 4$ result for the contradiction. 
\end{remark}

To prove Lemma~\ref{lem:almost}, we will show that if $w_{I}=z_{I}$ for all $I\neq [1, \dots ,1]$ and  $z_{[1, \dots ,1]}\neq w_{[1, \dots ,1]}$, then $z$ is a zero of an auxiliary set of polynomials denoted $B$.  We will then show that the zero set $\V(B)$ is contained in the union of Segre varieties.  Finally, Proposition~\ref{prop:ZZ block} provides the inclusion into $Z_{n}$.

%
%
%%%%%%%%%%%  Reduction to one variable  %%%%%%%%%%%%%%%%%%%%
\subsection{Reduction to one variable} \label{sec:reduction}
 Let $n\geq 4$. Suppose $z = z_{I}X^{I}$ and $w = w_{I}X^{I}$ are points in $\widehat{\V(HD)}$ are such that $z_{I}=w_{I}$ for all $I\neq [1, \dots ,1]$.
Both points are zeros of every polynomial in $HD$, but the only coordinate in which they can differ is $[1, \dots ,1]$. Now consider the coordinates $z_{I}$ ( $=w_{I}$) as fixed constants for all $I \neq [1, \dots ,1]$, and for $f\in HD$ define $f_{z}$ by the substitution $f(X^{[0, \dots ,0]}, \dots ,X^{[1, \dots ,1]}) \mapsto f(z_{[0, \dots ,0]}, \dots ,z_{[0,1, \dots ,1]},X^{[1, \dots ,1]})=: f_{z}(X^{[1, \dots, 1]})$. Let $HD_{[1, \dots ,1]}(z) = \{f_{z}\mid f\in HD\}$ denote the resulting set of univariate polynomials.  Then $z_{[1, \dots ,1]}$ and $w_{[1, \dots ,1]}$ are two (possibly different) roots of each univariate polynomial $f_{z}\in HD_{[1, \dots ,1]}(z)$.

\begin{lemma}
If $f\in HD$, then the corresponding polynomial $f_{z} \in HD_{[1, \dots ,1]}(z)$ is either degree 0, 1, or 2 in $X^{[1, \dots ,1]}$.
\end{lemma}
%revision suggested by referee
\begin{proof}
It suffices to prove the statement for $f\in S_{(2,2)}V_{1}^{*}\otimes S_{(2,2)}V_{2}^{*}\otimes S_{(2,2)}V_{3}^{*}\otimes S_{(4)} V_{4}^{*}\otimes \dots  \otimes  S_{(4)}V_{n}^{*}$.  Suppose for contradiction that $f$ has a monomial of the form $(X^{[1,\dots,1]})^{3}X^{[i_{1},\dots,i_{n}]}$.  The possible weights of this monomial are $(2+2i_{1},\dots,2+2i_{n})$ with $i_{j}\in \{0,1\}$. However the weight of every polynomial in $ S_{(2,2)}V_{1}^{*}\otimes S_{(2,2)}V_{2}^{*}\otimes S_{(2,2)}V_{3}^{*}\otimes S_{(4)} V_{4}^{*}\otimes \dots  \otimes  S_{(4)}V_{n}^{*}$ is of the form $(0,0,0,w_{4},\dots,w_{n})$, where $w_{i}$ are even integers with $|w_{i}|\leq 4$ for $4\leq i\leq n$, a contradiction since obviously $0\neq 2+2i_{1}$ for any $i_{1}\in \{0,1\}$. Therefore the degree of $f$ is less than $3$ in $X^{[1,\dots,1]}$.
\end{proof}

Now we know that $w_{[1, \dots ,1]}$ and $z_{[1, \dots ,1]}$ are both common zeros of univariate polynomials, all with degree 2 or less.  The fact that $w_{[1, \dots ,1]}$ and $z_{[1, \dots ,1]}$ are both common zeros of more than one univariate polynomial comes from the fact that we have required $n\geq 4$ otherwise there is only one polynomial and what we are about to do would be trivial.

A quadratic (not identically zero) in one variable has at most two solutions, and a linear polynomial (not identically zero) has at most one solution.  The only way then for us to have $w\neq z$ and $[w],[z]\in \V(HD)$ is if \em all \em of the linear polynomials were identically zero and if \em all \em of the quadratics were scalar multiples of each other.

Therefore, we need to study the points $[z]\in \V(HD)$ for which $HD_{[1,\dots,1]}(z)$ has dimension 1 or less.
%
%$f_{z}=\lambda_{f,g} g_{z}$ for all $f,g\in HD_{[1, \dots ,1]}(z)$, and some $\lambda_{f,g} \in \CC$.
%
Define polynomials $a_{f}$, $b_{f}$, and $c_{f}$ (which necessarily do not depend on $X^{[1, \dots ,1]}$) for each $f_{z}\in HD_{[1, \dots ,1]}(z)$ by 
\[
f_{z} = a_{f}(z) \left(X^{[1, \dots ,1]}\right)^{2} + b_{f}(z) \left(X^{[1, \dots ,1]}\right) + c_{f}(z).
\] 
The requirement that $HD_{[1,\dots,1]}(z)$ have dimension 1 or less implies the weaker (but still sufficient) condition that $z$ be a root of the polynomials 
\[B':=span\{a_{f} b_{g} - a_{g}b_{f}\mid f,g\in HD\}.\]

The polynomials in $B'$ have the property that if $h(z) \neq 0$ for a nonzero $h\in B'$, \ie $[z]\not\in \V(B')$, then there is a non-trivial pair of polynomials in $HD_{[1,\dots,1]}(z)$ that are not scalar multiples of each other, and thus the zero set of $HD_{[1,\dots,1]}(z)$ is a single point. In this case we must have $w_{[1, \dots ,1]}=z_{[1, \dots ,1]}$. If, however $h(z) =0$ for all $h\in B'$ (\ie $z\in \V(B')$), then it is possible that the polynomials in $HD_{[1, \dots ,1]}(z)$ have $2$ common roots.

Notice that $B'$ is not $\G$-invariant. Let $B := span\{ \left(\G\right).B'\}$ denote the corresponding $\G$-module.
%
%Since $\V(HD)$ is a $G$-variety, $[z]\in \V(HD)$ implies that $\overline{G.[z]} \subset \V(HD)$.  
%
If  $g.[z]\notin \V(B')$, then by our remarks above, $g.[z] \in Z_{n}$, and in particular, $[z]\in Z_{n}$ (because $Z_{n}$ is a $G$-variety). 
The following lemma allows us to compare $G$-orbits of points and the zero sets of arbitrary  sets of polynomials (not necessarily $G$-modules).
\begin{lemma}\label{lem:orbitspan}
 Let $z\in \PP V$, let $G\subset GL(V)$ be a group, and let $M\subset \Sym(V^*)$ be a collection of polynomials ($M$ is not necessarily a $G$-module). Then
\[G.z \subset \V(M) \text{ if and only if } z \in \V(span\{ G.M \})\]
\end{lemma}
\begin{proof}
 $G.z \subset \V(M)$ if and only if  $f(g.z)=0$ for all $g\in G$ and for all $f\in M$.  But from the definition of the $G$-action on the dual space, $f(g.z)= (g^{-1}.f)(z)$, so $f(g.z)=0$ for all $g\in G$ and  for every $f\in M$.  This happens if and only if $(g.f)(z) =0$ for all $g\in G$ and  for all $f\in M$, but, this is the condition that $z\in \V(span\{ G.M \})$.
\end{proof}
We apply Lemma~\ref{lem:orbitspan} to our setting; if $\left(\G\right).[z] \subset \V(B')$, then $[z]\in \V(B)$ (recall $B := span\{ \left(\G\right).B'\}$).
So, we need to look at the variety $\V(B)$. We conclude that our construction satisfies the property that if $[z]\in \V(HD)$ but $[z]\not\in \V(B)$, then $[z]\in Z_{n}$. 

We need to understand the types of points that can be in $\V(B)$ and the following proposition gives sufficient information about $\V(B)$.
\begin{prop}\label{prop:zerosB} Let $n\geq 4$ and let $B$ be the module of polynomials constructed above. Let $\mathcal{P}^{2}(\{1,\dots,n\})$ be the collection of all partitions of $\{1,\dots,n\}$ into mutually disjoint subsets of cardinality $2$ or less.  Then
 \[\V(B) \subset
\bigcup_{\{K_{1},\dots,K_{m}\}\in \mathcal{P}^{2}(\{1,\dots,n\})} 
 Seg\left( \PP V_{K_{1}} \tdt \PP V_{K_{m}} \right)  
\subset Z_{n}.\]
\end{prop}
\begin{proof}
Proposition \ref{prop:zerosB} will be proved in several parts.
In Lemma \ref{lem:zerosB1} we will find the module $S_{(4,1)}S_{(4,1)}S_{(4,1)}S_{(5)} \dots  S_{(5)}$
as a submodule of $B$ using the algorithm in Section~\ref{sec:lowering}.
In Lemma \ref{lem:big segre ideal} we will identify the zero set of this new module. In particular, we will show that
\[\V(S_{(4,1)}S_{(4,1)}S_{(4,1)}S_{(5)} \dots  S_{(5)}) = \bigcup_{\{K_{1},\dots,K_{m}\}\in \mathcal{P}^{2}(\{1,\dots,n\})} 
 Seg\left( \PP V_{K_{1}} \tdt \PP V_{K_{m}} \right) 
.\]
We prove this statement by induction on $n$, where we prove the base case $n=3$ in Lemma \ref{lem:singular base} and the induction step in Lemma \ref{lem:one piece n}.
Finally, each $\PP V_{K_{i}}$ is either a copy of $\PP^{1}\cong Z_{1}$ or $\PP^{3}\cong Z_{2}$ so we can apply Proposition~\ref{prop:ZZ block} to verify the inclusion \[Seg\left( \PP V_{K_{1}} \tdt \PP V_{K_{m}} \right) \subset Z_{n}.\]
\end{proof}

\begin{lemma}\label{lem:zerosB1} Suppose $n\geq 4$ and let $B$ be constructed as above. Then  
\[S_{(4,1)}S_{(4,1)}S_{(4,1)}S_{(5)} \dots  S_{(5)} \subset B.\]
\end{lemma}

\begin{proof} 
Here we have a subset of polynomials in $B$ in an explicit form, and we would like to identify $\G$-modules in $B$ from this information.
To do this we use the ideas presented in Section~\ref{sec:WeightSpaces} and particularly  the algorithm presented in Section~\ref{sec:lowering}.  It suffices to work first with $\SL(2)^{\times n}$-modules and later consider the permutations.

Suppose $f_{k_{1},k_{2},k_{3}}\in S_{(2,2)}V_{k_{1}}^{*}\otimes S_{(2,2)}V_{k_{2}}^{*} \otimes S_{(2,2)}V_{k_{3}}^{*}\otimes S_{(4)}V_{k_{4}}^{*}\otdot S_{(4)}V_{k_{n}}^{*}$ is a lowest weight vector.  Define $a_{k_{1},k_{2},k_{3}}$, $b_{k_{1},k_{2},k_{3}}$, $c_{k_{1},k_{2},k_{3}}$ by the equation $f_{k_{1},k_{2},k_{3}} = a_{k_{1},k_{2},k_{3}}(X^{[1, \dots ,1]})^{2} + b_{k_{1},k_{2},k_{3}}(X^{[1, \dots ,1]}) + c_{k_{1},k_{2},k_{3}}$.

For this proof, we introduce new notation. If $k_{1},k_{2},k_{3}$ are fixed, let $X^{I_{p,q,r}}$ denote the coordinate vector with $k_{1}=p, k_{2}=q,k_{3}=r$ and $k_{s}=0$ for $s\geq 4$. 

Since $f_{k_{1},k_{2},k_{3}}$ is a hyperdeterminant of format $2 \times 2 \times 2$ we find 
\[a_{k_{1},k_{2},k_{3}}= (X^{I_{0,0,0}})^{2}\]
\[ b_{k_{1},k_{2},k_{3}}=  -2 X^{I_{0,0,0}} \left(  X^{I_{1,0,0}}X^{I_{0,1,1}}  + X^{I_{0,1,0}}X^{I_{1,0,1}} + X^{I_{0,1,0}}X^{I_{1,1,0}} \right)  +4 X^{I_{1,0,0}}X^{I_{0,1,0}}X^{I_{0,0,1}} .
\]
The weight of $a_{k_{1},k_{2},k_{3}}$ is (up to permutation) $(-2,-2,-2,2, \dots ,2)$, where the $-2$'s actually occur at $\{k_{1},k_{2},k_{3}\}$. The weight of $b_{k_{1},k_{2},k_{3}}$ is (up to permutation) 

$ (-1,-1,-1,3, \dots ,3) $, where the $-1$'s actually occur at $\{k_{1},k_{2},k_{3}\}$.
Now consider
\[ h_{k_{1},k_{2},k_{3},j_{1},j_{2},j_{3}}= a_{k_{1},k_{2},k_{3}}b_{j_{1},j_{2},j_{3}}-a_{j_{1},j_{2},j_{3}}b_{k_{1},k_{2},k_{3}} \in B.\]
We notice that $h_{k_{1},k_{2},k_{3},j_{1},j_{2},j_{3}}$ can not have $k_{1},k_{2},k_{3}$ and $j_{1},j_{2},j_{3}$ all equal (this is the zero polynomial).  So either two, one or zero pairs of $i$'s and $j$'s match in the indices $k_{1},k_{2},k_{3}$ and $j_{1},j_{2},j_{3}$. Therefore  $h_{k_{1},k_{2},k_{3},j_{1},j_{2},j_{3}}$ can have $3$ different (up to permutation) weights, depending on how $k_{1},k_{2},k_{3}$ and $j_{1},j_{2},j_{3}$ match up.  The three possible weights of $h_{k_{1},k_{2},k_{3},j_{1},j_{2},j_{3}}$ are (up to permutation): $(-3,-3,1,1,5, \dots ,5)$, $(-3, 1,1,1 ,1,5, \dots ,5)$, or $(1, 1, 1, 1, 1, 1,5, \dots ,5)$.

In each case, apply the algorithm in Section~\ref{sec:lowering} and lower $h_{k_{1},k_{2},k_{3},j_{1},j_{2},j_{3}}$ to a nonzero vector with the lowest possible weight.  We did this calculation in Maple. The output in each case is a vector of weight (up to permutation) $(3,3,3,5, \dots ,5)$.  Next we use Remark~\ref{rmk:transition} to identify the module with lowest weight $(3,3,3,5, \dots ,5)$ as 
\[S_{(4,1)}S_{(4,1)}S_{(4,1)}S_{(5)} \dots  S_{(5)},\] and this must be a submodule of $B$.
\end{proof}

\begin{lemma}\label{lem:singular base} As sets in $\PP(V_{1}\otimes V_{2} \otimes V_{3})$
\begin{multline*}
 \V\left(S_{(4,1)}V_{1}^{*} \otimes S_{(4,1)}V_{2}^{*} \otimes S_{(4,1)}V_{3}^{*}\right) \\
  =  Seg(\PP (V_{1} \otimes V_{2})\times \PP V_{3})  \cup Seg(\PP (V_{1} \otimes V_{3})\times \PP V_{2})
     \cup Seg(\PP (V_{1}) \times(V_{1} \otimes V_{2}) .    \end{multline*}
\end{lemma}
\begin{proof} %only handle case n=3%
 The space $V_{1}^* \otimes V_{2}^* \otimes V_{3}^*$ has seven of orbits under the action of $\SL(2)^{\times 3}$ \cite[Example 4.5 p. 478]{GKZ}.  This gives rise to a list of normal forms, which we record below together with the respective $\G$-orbit closures to which they belong:
\begin{itemize}
\item The trivial orbit, $\emptyset$.
 \item $Seg(\PP V_{1} \times\PP V_{2} \times\PP V_{3})$ : Normal form $[x] = [a \otimes b \otimes c]$. 
,
 \item $ \tau(Seg(\PP V_{1} \times\PP V_{2} \times\PP V_{3}))_{sing} =  \mathfrak{S}_{3}.Seg(\PP (V_{1} \otimes V_{2})\times \PP V_{3})$: Normal form (up to permutation) $[x] = [a \otimes b \otimes c + a' \otimes b' \otimes c ]$.  This union of 3 irreducible varieties is the singular set of the next orbit.
 \item $\tau(Seg(\PP V_{1} \times\PP V_{2} \times\PP V_{3}))$: Normal form $[x] = [a \otimes b \otimes c + a' \otimes b \otimes c + a \otimes b' \otimes c + a \otimes b \otimes c']$.

 \item $\sigma(Seg(\PP V_{1} \times\PP V_{2} \times\PP V_{3}))$:  Normal form $[x] = [a \otimes b \otimes c + a' \otimes b' \otimes c']$.
\end{itemize}

The orbit closures are nested:
\begin{multline*}
\emptyset \subset Seg(\PP V_{1} \times\PP V_{2} \times\PP V_{3})
 \subset
 \tau(Seg(\PP V_{1} \times\PP V_{2} \times\PP V_{3}))_{sing}
\\
 \subset
 \tau(Seg(\PP V_{1} \times\PP V_{2} \times\PP V_{3}))
 \subset
 \sigma(Seg(\PP V_{1} \times\PP V_{2} \times\PP V_{3})) = \PP^{7}
. \end{multline*}

The lowest weight vector for $S_{(4,1)}S_{(4,1)}S_{(4,1)}$  is %(in the $n=3$ case)%
\begin{align*} 
f_{(4,1),(4,1),(4,1)}=(X^{[1,1,1]})^{2} \big(X^{[0,0,0]}(X^{[1,1,1]})^{2}  + 2X^{[1,0,1]}X^{[0,1,1]}X^{[1,1,0]} \\
-  X^{[1,1,1]} ( X^{[0,1,1]}X^{[1,0,0]} + X^{[1,0,1]}X^{[0,1,0]}+X^{[1,1,0]}X^{[0,0,1]} ) \big)
.\end{align*}

We took a generic point $x\in \tau(Seg(\PP V_{1} \times\PP V_{2} \times\PP V_{3}))_{sing}$ and evaluated $f_{(4,1),(4,1),(4,1)}(x)=0$.  So therefore  $\tau(Seg(\PP V_{1} \times\PP V_{2} \times\PP V_{3}))_{sing} \subset \V(S_{(4,1)}S_{(4,1)}S_{(4,1)})$.  We could also conclude this without a calculation by noticing that any point of the form $ [a \otimes b \otimes c + a' \otimes b' \otimes c ]$ lives in $\PP(V_{1}\otimes V_{2}\otimes \widehat{c})$, where $\widehat{c}$ is the line through $c$. But every point in this space is a zero of $S_{(4,1)}S_{(4,1)}S_{(4,1)})$ because $S_{(4,1)}(\widehat{c})^{*}=0$.

Next, we show that the other two varieties are not in $\V(S_{(4,1)}S_{(4,1)}S_{(4,1)})$.  The varieties are nested, so consider the point $[x] = \big[X^{[1,1,1]} + X^{[0,1,1]} + X^{[1,0,1]} + X^{[1,1,0]}\big] \in \tau(Seg(\PP V_{1} \times\PP V_{2} \times\PP V_{3}))$.  But $f_{(4,1),(4,1),(4,1)}(x)  = 2 \neq 0$, so the other two varieties are not in $\V(S_{(4,1)}S_{(4,1)}S_{(4,1)})$.
Since we have considered all possible normal forms, we are done.\end{proof}

\begin{obs}
 All the linear spaces on $Seg(\PP V_{K_{1}} \tdt \PP V_{K_{m}})$ are (up to permutation) contained in one of the form $V_{K_{1}}\otimes \widehat{v_{K_{2}} } \otdot \widehat{v_{K_{m}}}$, where $\widehat{v_K}$ denotes the line through $ v_{k_{1}}\otdot v_{k_{s}}$ in $V_{K}$.
\end{obs}
Let $\mathcal{P}_{p,q}(\{n_{1},\dots,n_{p+q}\})$ denote the set of partitions of $\{n_{1},\dots,n_{p+q}\}$ into two disjoint sets of cardinality $p$ and $q$.

Consider the $\G$-module $\tilde{B} = S_{(4,1)}S_{(4,1)}S_{(4,1)}S_{(5)}\dots S_{(5)}$ that has $n-3$ copies of $S_{(5)}$. We write in $\tilde{B}$ in full detail as
\[\tilde{B} = 
\bigoplus_{\{\{k_{1},k_{2},k_{3}\},\{k_{4},\dots,k_{n}\}\}\in\mathcal{P}_{3,n-3}(\{1,\dots,n\})} S_{(4,1)}V_{k_{1}}^{*} \otimes S_{(4,1)}V_{k_{2}}^{*} \otimes S_{(4,1)}V_{k_{3}}^{*} \otimes S_{(5)} V_{k_{4}}^{*}\otdot S_{(5)} V_{k_{n}}^{*},
.\] 
Let $\tilde{B}_{k}$ denote the $\SL(2)^{n-1}\ltimes \mathfrak{S}_{n-1}$ module
\[\tilde{B}_{k} = 
\bigoplus_{\{\{k_{1},k_{2},k_{3}\},\{k_{4},\dots,k_{n-1}\}\}\in\mathcal{P}_{3,n-4}(\{1,\dots,n\}\setminus\{k\})}
S_{(4,1)}V_{k_{1}}^{*} \otimes S_{(4,1)}V_{k_{2}}^{*} \otimes S_{(4,1)}V_{k_{3}}^{*} \otimes S_{(5)} V_{k_{4}}^{*}\otdot S_{(5)} V_{k_{n-1}}^{*}
.\]
Notice that $\tilde{B} \sum_{i=1}^{n} \tilde{B}_{i} \otimes S_{(5)}V_{i}^{*}$. In other words the $\G$-module $\tilde{B}$ is constructed as the non-redundant sum over permutations of augmented $\SL(2)^{\times n-1}$-modules.

We want to understand the zero set of this module $\tilde{B}$, and we do this in the next two lemmas by mimicking what we did for Proposition~\ref{prop:segre ideal}.  We also point out that while notationally more complicated, the resulting Lemma~\ref{lem:big segre ideal} is essentially the same idea as Proposition~\ref{prop:segre ideal}.

\begin{lemma}\label{lem:one piece n}  Suppose $n\geq 4$ and let $\tilde{B}$ and $\tilde{B}_{k}$ be as above. If
\[
\V\left(\tilde{B}_{k}\right)  = 
\bigcup_{\{K_{1},\dots,K_{m}\}\in \mathcal{P}^{2}(\{1,\dots,n\}\setminus\{k\}) } 
Seg\left(\PP V_{K_{1}} \times \PP V_{K_{2}} \tdt \PP V_{K_{m}} \right) 
,\]  
then
 \begin{equation}\label{eq:one piece n} \V\left(\tilde{B}_{k} \otimes S_{(5)}V_k^{*} \right)  
= \bigcup_{\{K_{1},\dots,K_{m}\}\in \mathcal{P}^{2}(\{1,\dots,n\}\setminus\{k\}) } 
 Seg\left(\PP V_{K_{1}\cup\{k\}} \times \PP V_{K_{2}} \tdt \PP V_{K_{m}} \right) 
.\end{equation}
\end{lemma}
\begin{proof}
Apply the Augmentation Lemma~\ref{lem:augment} to the left hand side of \eqref{eq:one piece n}.  It remains to check that 
\[
\bigcup_{L\subset \V(\tilde{B}_{k})}\PP(L\otimes V_{k})=
\bigcup_{\{K_{1},\dots,K_{m}\}\in \mathcal{P}^{2}(\{1,\dots,n\}\setminus\{k\}) } 
Seg\left(\PP V_{K_{1}\cup\{k\}} \times \PP V_{K_{2}} \tdt \PP V_{K_{m}} \right) 
,\]
where $L\subset \V(\tilde{B}_{k})$ are linear spaces.
Because of symmetry and our hypothesis, there is only one type of linear space to consider, $V_{I_{1}} \otimes \widehat{v_{I_{2}}} \otdot \widehat{v_{I_{m}}} \otimes V_{k} = V_{I_{1}\cup \{k\}}\otimes \widehat{v_{I_{2}}} \otdot \widehat{v_{I_{m}}} $.  It is clear that  each of these linear spaces is on one of the Segre varieties on the right hand side of \eqref{eq:one piece n}, and moreover every point on the right hand side of \eqref{eq:one piece n} is on one of these linear spaces.
\end{proof}

\begin{lemma}\label{lem:big segre ideal}  Suppose $n\geq 4$ and let $\mathcal{P}^{2}(\{n_{1},\dots,n_{n}\})$ denote the collection of all partitions of $\{n_{1},\dots,n_{n}\}$ into mutually disjoint subsets of cardinality $2$ or less. Then
\[ \V\left( S_{(4,1)}S_{(4,1)}S_{(4,1)}S_{(5)}\dots S_{(5)} \right) = 
\bigcup_{\{K_{1},\dots,K_{m}\}\in \mathcal{P}^{2}(\{1,\dots,n\}) } 
Seg\left( \PP V_{K_{1}} \tdt \PP V_{K_{m}} \right) .\]
\end{lemma}
%\begin{remark}
%This may be condensed as follows.  If $n =2k+1$, then  
%\[ \V\left(\tilde{B}\right) = \bigcup_{\begin{array}{c}|I_s| = 2 \\ 1\leq s \leq 2k \end{array} , |I_{n}|=1} \left( Seg\left( \PP V_{I_{1}} \tdt \PP V_{I_{2k}} \times \PP V_{I_{n}} \right) \right).\]
% If $n =2k$, then  
%\[ \V\left(\bigoplus_{|I|=n}\left(S_{(4,1)}V_{\{k_{1},k_{2},k_{3}\}}^{*} \otimes S_{(5)} V_{I\backslash \{k_{1},k_{2},k_{3}\}}^{*}\right)\right) = \bigcup_{\begin{array}{c}|I_s| = 2 \\ 1\leq s \leq 2k \end{array}} \left( Seg\left( \PP V_{I_{1}} \tdt \PP V_{I_{2k}}  \right) \right).\]
%\end{remark}
%
\begin{proof}
Proof by induction. The base case is Lemma~\ref{lem:singular base}. For the induction step, 
 use Lemma~\ref{lem:one piece n}. We need to show that 
\begin{align*}
\bigcap_{k=1}^{n}  \left(
\bigcup_{\{K_{1},\dots,K_{m}\}\in \mathcal{P}^{2}(\{1,\dots,n\}\setminus\{k\}) } 
 Seg\left(\PP V_{K_{1}\cup\{k\}} \times \PP V_{K_{2}} \tdt \PP V_{K_{m}} \right) 
\right)\\ =
\bigcup_{\{K_{1},\dots,K_{m}\}\in \mathcal{P}^{2}(\{1,\dots,n\}) }
Seg\left( \PP V_{K_{1}} \tdt \PP V_{K_{m}} \right)
.\end{align*}
It suffices to check that
\begin{align*} Seg\left(\PP V_{K_{1}\cup\{k\}} \times \PP V_{K_{2}} \times  \PP V_{K_{3}} \tdt \PP V_{K_{m}} \right) \\
\cap Seg\left(\PP V_{K_{1}} \times \PP V_{K_{2}\cup\{k\}} \times  \PP V_{K_{3}} \tdt \PP V_{K_{m}} \right) \\
=Seg\left(\PP V_{K_{1}} \times \PP V_{K_{2}} \times \PP V_{k} \times  \PP V_{K_{3}} \tdt \PP V_{K_{m}} \right).\end{align*}
This is equivalent to checking that for any vector spaces $V_{1},V_{2},V_{3}$ that 
\begin{align*} Seg\left(\PP (V_{1}\otimes V_{2}) \times \PP V_{3}\right) 
\cap Seg\left( \PP V_{1}\times \PP (V_{2}\otimes V_{3})  \right)
=Seg\left( \PP V_{1}\times \PP V_{2}\times \PP V_{3})  \right).\end{align*}
In this case, let $[T] \in Seg\left(\PP (V_{1}\otimes V_{2}) \times \PP V_{3}\right) 
\cap Seg\left( \PP V_{1}\times \PP (V_{2}\otimes V_{3})  \right)$.  Then, viewed as a map $T:(V_{1}\otimes V_{2})^{*}\rightarrow V_{3}$, the image of $T$ must be one dimensional, thus $[T]\in \PP(V_{1}\otimes V_{2}\otimes V_{3}')$ where $V_{3}'\subset V_{3}$ is a one dimensional subspace.  
By the same argument using the other Segre variety in the intersection, $[T]\in\PP(V_{1}\otimes V_{2}'\otimes V_{3})$, where $V_{2}'\subset V_{2}$ is a one dimensional subspace.  So $T\in \PP (V_{1}\otimes V_{2}' \otimes V_{3}')$, but this is a linear space on $Seg\left( \PP V_{1}\times \PP V_{2}\times \PP V_{3})  \right)$, so we are done. 
\end{proof}
We conclude this section by pointing out that we have established all of the ingredients for the proof of Lemma~\ref{lem:almost}.

%%%%%%%%%%%%%%  ^^^^^^^^^  Almost ^^^^^^^^^     %%%%%%%%%%%%%%%%%%%

\section{Proof of Theorem~\ref{thm:Main Theorem}} \label{sec:theorem}
The outline of the proof is the following. Proposition~\ref{prop:one way} says that $Z_{n}\subseteq \V(HD)$.  
To show the opposite inclusion, we work by induction. In the cases of $n=3,4$, the (stronger) ideal-theoretic version of Theorem~\ref{thm:Main Theorem} was proved with the aid of a computer in \cite{HSt}. Since the theorem is already proved for the cases $n=3,4$ we will assume $n\geq 5$.
The induction hypothesis is that $\V(HD_{i}) \simeq Z_{(n-1)}$. We need to show that given a point $[z] \in \V(HD)$, that $[z]\in Z_{n}$, \ie that there exists a matrix $A$ so that $\varphi([A,t]) = [z]$. The key tools we use in this proof are Proposition~\ref{prop:no almost} and Lemma~\ref{lem:characterization}.

We will work on a preferred open set $U_{0}= \{ [z] = [z_{I}X^{I}]\in\PP( V_{1}\otdot V_{n})\mid z_{[0, \dots ,0]} \neq 0 \}$.  Choosing to work on this open set is no loss of generality because of the following
\begin{lemma}
Let $U_{0}= \{ [z] = [z_{I}X^{I}]\in\PP( V_{1}\otdot V_{n})\mid z_{[0, \dots ,0]} \neq 0 \}$. Then $\V(HD) \cap U_{0} \subset Z_{n}$ implies that $\V(HD) \subset Z_{n}$.
\end{lemma}
\begin{proof}
The result follows from the facts that $Z_{n}$ and $\V(HD)$ are $\G$-invariant, and $\left(\G\right).U_{0} = \PP(V_{1}\otdot V_{n})$. 
\end{proof}
Moreover, it suffices to work on the following section of the cone over projective space, $\{z=z_{I}Z^{I}\in V_{1}\otdot V_{n} \mid z_{[0,\dots,0] }=1\}$, because afterwards we can rescale everything to get the result on the whole open set $U_{0}$.  

Suppose we take a point in the zero set (as described by Lemma~\ref{lem:characterization})
\[[z] \in \V(HD)= \bigcap_{i=1}^{n} \bigcup_{L^i\subset \V(HD_{i})} \PP (L^i \otimes  V_{i}).\] 
Since $[z]$ is fixed, we can also fix a single $L^{i}$ for each $i$ so that $[z] \in  \bigcap_{i=1}^{n}\PP (L^i \otimes  V_{i})$. 
Work in our preferred section of the cone over projective space and write $n$ different expressions for the point $z$ (one for each $i$):
\[
z= z_{I} X^{I}  = \eta^i \otimes x_{i}^0 + \nu^i \otimes x_{i}^{1} ,
\]
where $[\eta^i],[\nu^i] \in L^i \subset \V(HD_{i})$. (These expressions are possible because each $V_{i}$ is $2$ dimensional.)  Choosing $z_{[0,\dots,0]} =1$ also implies that $\eta^{i}_{[0,\dots,0]}=1$.  The induction hypothesis says that  $Z_{(n-1)}\simeq\V(HD_{i})$ for $1\leq i\leq n$. So each $\eta^{i}$ satisfies $\varphi([A^{(i)},1]) = \eta^{i}$ for a symmetric matrix $A^{(i)}\in S^{2}\CC^{n-1}$.
For each $0\leq j\leq n$ denote by $\A^{j}$ the following subset of matrices
\[
\A^{j} = \{A \in S^{2}\CC^{n} \mid \Delta_{I}(A) = z_{I} \text{ for all } I = [i_{1},\dots,i_{n}] \text{ with } i_{j}= 0 \}
.\]
Each matrix  $A\in \A^{j}$ has the property that the principal submatrix of $A$ formed by deleting the $j^{th}$ row and column maps to $\eta^{j}$ under the principal minor map.  Thus each $A\in \A^{j}$ is a candidate matrix that might satisfy $\varphi([A,1])=[z]$, however we don't know if such a matrix will have a submatrix that maps to the other $\eta^{i}$'s. We claim that there is at least one matrix that satisfies all of these conditions.

\begin{lemma}\label{lem:nonempty} $\cap_{i=1}^{n}\A^{i}$ is non empty.
\end{lemma}
\begin{proof} By the induction hypothesis, each $\A^{i}$ is non-empty. Assume $\cap_{i=2}^{n}\A^{i}$ is non-empty. We show that if $A\in \cap_{i=2}^{n}\A^{i}$ then $A\in \A^{1}$.  The same argument we use will also prove that if $A\in \cap_{i=3}^{n}\A^{i}$, then $A\in \A^{1}$, and so on, so it suffices to check the last, most restrictive case.  Also because of the $\mathfrak{S}_{n}$ action, we don't have to repeat the proof for every permutation.

If $A\in \cap_{i=2}^{n}\A^{i}$, then $\Delta_{I}A = z_{I}$ for all $I \neq [0,i_{2},\dots,i_{n}]$ with $|I|\leq n-2$.  The only possible exception we could have is for $\Delta_{[0,1,\dots,1]}$ might not be equal to $z_{[0,1,\dots,1]}$.
Let $A'$ denote the principal submatrix of $A$ formed by deleting the $1^{st}$ row and column of $A$.
Now since $n\geq 5$, $|I| \geq 3$, $A'$ is at least as large as $4\times 4$, and we have determined that all of the principal minors of $A'$ except possibly the determinant agree with a fixed point $\eta^{1} \in \V(HD_{1})$ (in other words $\Delta(A')_{I}= \eta^{i}_{I}$ for all $I\not=[1,\dots,1]$), so we can apply Proposition~\ref{prop:no almost} to conclude that the determinant of $A'$ also agrees with $\eta^{1}$ (\ie $\Delta_{[1,\dots,1]}(A') = \eta^{1}_{[1,\dots,1]}$). Therefore any such $A$ must have $\Delta_{[0,1,\dots,1]}(A) = z_{[0,1,\dots,1]}$, and we have shown $A\in \A^{1}$.
\end{proof}
Lemma~\ref{lem:nonempty} above proves the existence of a symmetric matrix $A$ such that $\Delta_{I}(A) = z_{I}$ for all $I \neq [1,\dots,1]$. Then since both $z_{I}X^{I}$ and $\Delta_{I}(A) X^{I}$ are points in $\V(HD)$, Proposition~\ref{prop:no almost} implies that $\Delta_{[1,\dots,1]}(A) = z_{[1,\dots,1]}$, and this finish the proof of the main theorem.

\begin{remark}[\textbf{Building a matrix}]
Note that when $n\geq 4$, the proof we gave can be used also to construct a symmetric matrix whose principal minors are prescribed by a point $z\in \widehat{\V(HD)} \cap \{z\mid z_{[0,\dots,0]} \neq 0\}$.  The entries of $z$ corresponding to $1\times 1$ and $2 \times 2$ principal minors determine a large finite set $\A$ of candidate matrices that could map to $z$.  Restrict the set $\A$ to only those matrices whose $3\times 3$ principal minors agree with the corresponding entries of $z$, \ie keep only the matrices $A$ so that $\Delta(A)_{I}=z_{I}$ for all $|I| \leq 3$.  We claim that the remaining set of matrices all map to $z$ under the principal minor map. If $A$ is such that all of the $3\times 3$ principal minors agree with $z$, then Proposition~\ref{prop:no almost} implies that each $4\times 4$ principal minor of $A$ must agree with $z$ also. Iterate this argument to imply that all of the principal minors of $A$ must agree with $z$.
\end{remark}

%%%%%%%%%%%%%%  ^^^^^^^^^ Proof of Main Theorem  ^^^^^^^^^     %%%%%%%%%%%%%%%%%%%

\section*{Acknowledgments }
The author would like to thank J.M. Landsberg for suggesting this problem as a thesis topic and for his endless support and advice along the way.  We thank the two anonymous reviewers who read the first draft of this paper as well as the third reviewer who read the second draft for their numerous useful suggestions for revision. We also thank Shaowei Lin, Linh Nguyen, Giorgio Ottaviani,  Bernd Sturmfels, and Zach Teitler for useful conversations. Shaowei Lin pointed out the reference \cite{Nanson}.  Bernd Sturmfels suggested the addition of Corollary~\ref{cor}.

\bibliographystyle{amsplain}
\bibliography{master_bibdata}

\end{document}